\documentclass[reqno]{amsart}
\usepackage{amsmath, amssymb, amsthm, geometry, enumerate, graphicx, amsfonts, hyperref, mathrsfs}
\hypersetup{colorlinks=true,linkcolor=red, anchorcolor=blue, citecolor=blue, urlcolor=red, filecolor=magenta, pdftoolbar=true}
\theoremstyle{plain}
\newtheorem{thm}{\it Theorem}[section]
\newtheorem{prop}[thm]{\it Proposition}
\newtheorem{cor}[thm]{\it Corollary}
\newtheorem{lem}[thm]{\it Lemma}
\theoremstyle{remark}
\newtheorem{defn}[thm]{Def{}inition}
\newtheorem{rem}[thm]{Remark}
\newtheorem{exa}[thm]{Example}
\numberwithin{equation}{section}
\usepackage{enumerate}
\usepackage{relsize}
\allowdisplaybreaks
\begin{document}

\title [Discrete            Vector-Valued  Nonuniform Gabor Frames]{Discrete   Vector-Valued  Nonuniform Gabor Frames}

\author[Lalit   Kumar Vashisht]{Lalit  Kumar  Vashisht}
\address{{\bf{Lalit  Kumar  Vashisht}}, Department of Mathematics,
University of Delhi, Delhi-110007, India.}
\email{lalitkvashisht@gmail.com}

\author[Hari Krishan Malhotra]{Hari Krishan Malhotra  }
\address{ \textbf{Hari Krishan Malhotra}, Department of Mathematics,
University of Delhi, Delhi-110007, India.}
\email{maths.hari67@gmail.com}

\begin{abstract}
Gabor frames have interested many mathematicians and physicists due to their potential applications in time-frequency analysis, in particular,  signal processing.  A Gabor system is a collection of vectors which is obtained by applying  modulation and shift operators to  non-zero functions in  signal spaces. In many applications, for example, signal processing related to Gabor systems,  the corresponding shifts  may not be uniform.  That  is,  the set associated with shifts  may not be a group under usual addition.  We analyze discrete vector-valued  nonuniform Gabor frames (DVNUG frames, in short) in discrete vector-valued  nonuniform   signal spaces, where  the indexing  set associated with shifts  may not be a subgroup of real numbers  under usual addition, but a spectrum which is  based on the theory of spectral pairs. First, we give necessary and sufficient conditions for the existence of DVNUG  Bessel sequences  in  discrete vector-valued  nonuniform  signal spaces in terms of Fourier transformations of the modulated window sequences. We provide a  characterization of DVNUG frames in  discrete  vector-valued nonuniform signal spaces. It is shown that DVNUG  frames are stable under small perturbation of window sequences associated with given  discrete vector-valued  nonuniform Gabor systems. We observed that the  arithmetic mean  sequences  associated with  window sequences of a given  DVNUG frame collectively  constitutes a discrete nonuniform Gabor frame.
Finally, we discuss an interplay between window sequences of DVNUG  systems and their  corresponding  coordinates.
\end{abstract}


\keywords{Discrete frame, Bessel sequence, Gabor frames, spectral pair, perturbation.\\
The research of  first author is  supported by the Faculty Research Programme  Grant-IoE, University of Delhi, Delhi-110007,  India (Grant No.: IoE/FRP/PCMS/2020/27). The second author is supported by  the University Grant Commission (UGC), India (Grant No.: 19/06/2016(i)EU-V)}

\maketitle

\baselineskip15pt
\section{Introduction}
Gabor frames, also known as Weyl-Heisenberg frames,  originate in quantum mechanics by Neumann \cite{Neumann},  and in communication theory by the idea of   Gabor \cite{Gabor} to represent a complicated signal in terms of elementary functions. Duffin and  Schaeffer, in \cite{DS}, introduced the notion of frames in separable Hilbert spaces during their deep study  of non-harmonic Fourier series.
For fixed positive real numbers $a$, $b$, and a non-zero function $\psi$ (called \emph{window function}) in $L^2(\mathbb{R})$, a \emph{Gabor system} is the collection of functions
\begin{align*}
\mathcal{G}(a, b, \psi): = \{E_{mb}T_{na}\psi\}_{m, n \in \mathbb{Z}},
\end{align*}
where $E_{b}\psi(\bullet) = e^{2 \pi i b \bullet} \psi(\bullet)$ and $T_{a}\psi(\bullet) = \psi(\bullet -a)$ are modulation and translation operators, respectively, that act unitarily on the signal space $L^2(\mathbb{R})$. In time-frequency analysis, one important direction is to determine scalars   $a$, $b$ and window functions $\psi$ such that every function $f$ in $L^2(\mathbb{R})$ can be decomposed and analyzed  in terms of systems of time-frequency shift $\mathcal{G}(a, b, \psi)$.
 Frame theory  provides us to describe when reproducing formulas  for $f$  of the form
\begin{align*}
f = \sum_{m,n \in \mathbb{Z}} c_{m,n} E_{mb}T_{na}\psi, \  \text{and analysis of function} \ f \mapsto \{\langle f,  E_{mb}T_{na}\psi\}_{m, n \in \mathbb{Z}}
\end{align*}
are possible. The Gabor system  $\mathcal{G}(a, b, \psi)$ is called a \emph{Gabor frame} for $L^2(\mathbb{R})$, if for all $f$ in $L^2(\mathbb{R})$,
\begin{align*}
\beta_o \|f\|^2 \leq \sum_{m,n \in \mathbb{Z}} |\langle f,  E_{mb}T_{na}\psi\rangle|^2 \leq \gamma_o \|f\|^2
\end{align*}
 holds for some finite positive real numbers $\beta_o$ and $\gamma_o$. For more than three  decades, the Gabor frames have been acknowledged as a very powerful tool in time-frequency analysis, see \cite{Groch, H89}. To be precise, the study of Gabor frames is an integral part of time-frequency analysis. Further, the  theory of Gabor frames has become a most interesting area for researchers due to its wide applicability in engineering sciences and pure mathematics, see  \cite{CK, ole, Classen, DGM, DabI, FZUS, Heil20, JaanI, PalleyJ} and many references therein.

Recently, Gabardo and Nashed \cite{GN1} introduced \emph{nonuniform multiresolution analysis} (NUMRA) and constructed
 nonuniform wavelets in the signal space $L^2(\mathbb{R})$, in which the translation set may not be a subgroup of real numbers under usual addition, but a spectrum which is based on the theory of  spectral pairs (see Definition \ref{SPDefn}). Gabardo and Yu \cite{GN2} continued the same study and gave
  characterization for nonuniform wavelets associated with a  NUMRA.
\begin{defn}\cite[p.\,800]{GN2}
	Let $N\ge 1$  be a positive integer  and $r$ be a fixed  odd integer coprime with $N$ such that $1\le r\le 2N-1$. A set of non-zero functions $\{\psi_1,\psi_2,\dots,\psi_k\}\subset L^2(\mathbb{R})$ is said to be \emph{nonuniform orthonormal wavelets} for $L^2(\mathbb{R})$, if the collection of functions,
	\begin{align*}
	\Big\{(2N)^{\frac{j}{2}}\psi_l((2N)^{j} \cdot-\lambda) \Big\}_{j\in \mathbb{Z},\lambda \in \Lambda \atop l =1,2,\dots,k}
	\end{align*}
	forms a complete orthonormal system  for $L^2(\mathbb{R})$, where $\Lambda =\left\{0,\frac{r}{N}  \right\} +2\mathbb{Z} $.
\end{defn}
They characterized  nonuniform wavelets associated with NUMRA for the space $L^2(\mathbb{R})$. Gabardo and Nashed, in  \cite{JN},  proved the analogue of Cohen's condition for NUMRA. We refer to \cite{GN2,  HK21xx, YuG} for fundamental results of  wavelets of nonuniform systems. A characterization  of scaling functions  of  NUMRA can be found in  \cite{HK}. The authors in \cite{HK20} proved the unitary extension principle and  the  oblique extension principle   for construction of  multi-generated nonuniform tight wavelet frames for the signal space $L^2(\mathbb{R})$.

Frazier, in \cite{frazier}, studied   $P^{th}$-stage discrete wavelets in both finite dimensional $\ell^2(\mathbb{Z}_N)$ and  infinite dimensional $\ell^2(\mathbb{Z})$ discrete signal space. In this direction, Malhotra and Vashisht, in  \cite{HK21}, studied construction of  $P^{th}$-stage  nonuniform discrete wavelet frames in discrete nonuniform signal spaces.
 Lopez and Han \cite{lopez} characterized discrete tight Gabor frames, dual frame pairs and orthogonal Bessel sequences in  $\ell^2(\mathbb{Z}^d)$. Xu, Lu and Fan in \cite{Xu},  constructed $J^{th}$-stage discrete periodic wave packet frames  in $\ell^2(\mathbb{Z}_N)$. Gressman \cite{gressman} studied wavelets on integers.  Lu and Li \cite{lu} studied the frame properties of generalized shift-invariant system in $\ell^2(\mathbb{Z}^d)$.  Discrete Gabor analysis and discrete wavelet analysis  attracted  many mathematicians and physicist  in recent years due to its potential applications in signal processing, see  \cite{CK, ole, DVII, JaanII,  SLiI, LiDI, LiDII, riol} and many references therein.
 We recall that there are  differences between discrete Gabor analysis and  the continuous Gabor one. More precisely,  Heil showed in \cite{HeilIII}  that Gabor frames in the continuous case are bases only if they are generated by functions that have poor decay or they are not smooth. But, in discrete case, it is possible   to construct Gabor frames that are bases and are generated by functions with good decay.

 Motivated by  fundamental work  of Frazier; Gabardo and Nashed; and potential applications of discrete Gabor frames in both pure and engineering  sciences, we study the frame properties of discrete  vector-valued nonuniform Gabor frames of the form
 \begin{align*}
\mathcal{G}(\Lambda,  T_1, T_2, \bold{W}_j) :=\big\{E_{\frac{m}{M}}R_{2N\lambda}\bold{W}_j :\lambda\in \Lambda,m\in T_1,j\in T_2\big\}
\end{align*}
in the discrete vector-valued nonuniform   signal space $\ell^2(\Lambda,\mathbb{C}^S)$, see Definition \ref{DDEFNI}. Here, $R_{(.)}$ and $E_{(.)}$, respectively,  are shift operator and modulation operator on  $\ell^2(\Lambda,\mathbb{C}^S)$, and $\bold{W}_j$ are window sequences in $\ell^2(\Lambda,\mathbb{C}^S)$.  Notable contributions in the present work include  a necessary and sufficient condition for the existence of discrete vector-valued nonuniform Gabor Bessel sequences; and  characterization of discrete vector-valued nonuniform Gabor frames.  It is also observed  that the  arithmetic mean  sequences  associated with  window sequences of a given discrete vector-valued nonuniform Gabor  frame collectively  constitutes a discrete nonuniform Gabor frame.

\vspace{10pt}

\textbf{Structure of paper and contribution:} In order to make the paper self-contained, Section \ref{Sect2} gives notation  and basic results related to Bessel sequences, discrete Hilbert frames in separable Hilbert spaces; and basics on  discrete vector-valued nonuniform   signal spaces. It also contains elementary properties of the Fourier transform on   discrete  vector-valued nonuniform   signal spaces. In Section \ref{Sect3}, we  give Bessel sequences and frames of discrete vector-valued nonuniform Gabor systems, where the translation parameter is based on the theory of spectral pairs, see Definition \ref{DefnnonI}. Theorem \ref{th1} and Theorem \ref{th2} provides, respectively,  a necessary and sufficient condition for the existence of discrete vector-valued nonuniform Gabor Bessel sequences in terms of Fourier transformations of the modulated window sequence. A characterization of discrete vector-valued nonuniform Gabor frame can be found in Theorem \ref{main1}, which based on a  fundamental tool developed in Lemma \ref{Lemma1}. Theorem \ref{per} in Section \ref{sect4}, shows that discrete vector-valued nonuniform Gabor frames are stable under small perturbation. In Section \ref{sect5}, we discuss  an interplay between window sequences of discrete vector-valued nonuniform Gabor system  and their corresponding coordinates. Theorem \ref{InterthI} shows that the arithmetic mean  sequences  associated with  window sequences of a given   discrete vector-valued nonuniform Gabor frame collectively  constitutes a discrete nonuniform Gabor frame.
Finally, in Theorem \ref{120}, we discuss relationship between the window sequences of discrete vector-valued nonuniform Gabor system  and its corresponding coordinates. Examples and counterexamples are also given to illustrate our  results.

\section{Technical Tools}\label{Sect2}
This section provides basic notation and results about Bessel sequences, discrete frames and basic tools related with nonuniform signal spaces.  We begin with the definition of a discrete Hilbert frame in separable Hilbert spaces.
\subsection{Discrete Hilbert Frames} Let $\mathcal{H}$ be a real or complex separable, finite or infinite dimensional Hilbert space with respect to an inner product $\langle \cdot,   \cdot\rangle$. The norm induced  by $\langle \cdot,   \cdot\rangle$ is given by $||g||_{\mathcal{H}}=\sqrt{\langle g,g\rangle}$, $g \in \mathcal{H}$. A  collection  $\{g_k\}_{k\in \mathbb{I}}$ of vectors in  $\mathcal{H}$, where $\mathbb{I}$ is a countable index set, is called  a \emph{discrete Hilbert frame} (or simply \emph{discrete frame}) for $\mathcal{H}$ if for some  positive real constants $a_0$ and $b_0$, the following  holds
\begin{align}\label{2020}
a_o \|g\|_{\mathcal{H}}^2\leq  \sum_{k \in \mathbb{I}} |\langle g, g_k\rangle|^2 \leq b_o \|g\|_{\mathcal{H}}^2, \ g \in \mathcal{H}.
\end{align}
 The constants $a_0$ and $b_0$ are called \emph{lower} and \emph{upper frame bounds} of $\{g_k\}_{k\in \mathbb{I}}$,  respectively. If  $a_0=b_0$, then $\{g_k\}_{k\in \mathbb{I}}$ is called \emph{tight discrete Hilbert frame} and if $a_0=b_0=1$ then it is called \emph{Parseval discrete Hilbert frame} for $\mathcal{H}$. If in \eqref{2020} only right hand side inequality holds then $\{g\}_{k\in \mathbb{I}}$ is called \emph{Bessel sequence} with Bessel bound $b_0$. Associated with a Bessel sequence $\{g_k\}_{k\in \mathbb{I}}$, the map $\Theta:\mathcal{H}\to \mathcal{H}$ given by $\Theta:g\to \sum_{k\in\mathbb{I}}\langle g, g_k\rangle g_k  $  is called \emph{frame operator} of $\{g_k\}_{k\in \mathbb{I}}$. It is bounded and linear; and invertible on $\mathcal{H}$ if $\{g_k\}_{k\in \mathbb{I}}$ is a frame for $\mathcal{H}$. Thus, if $\{g_k\}_{k\in \mathbb{I}}$ is a frame for $\mathcal{H}$, then each vector $g$ of $\mathcal{H}$ can be expressed as series, not necessarily unique,  $g= \Theta \Theta^{-1} g=\sum_{k\in \mathbb{I}}\langle \Theta^{-1} g,g_k\rangle g_k$. This series expansion is useful in signal  reconstruction and their analysis; and  take over many areas of analysis for hot topic for research.  Here, we would like to mention  that frames are also used in, however this list in incomplete: sampling theory \cite{AkramI},  signal processing \cite{BCGLL, DVIII, DVIV}, iterated function system \cite{VD3}, quantum physics \cite{JVash},  wavelet theory and multiresolution analysis \cite{rayan, Bhan, Mallat, Meyer,  RZ1, RZ3}, also see many references therein.


The following  basic results about Bessel sequence will be used in sequel.
 \begin{thm}\label{nh1}\cite[p.\,75]{ole}
 	A sequence  $\{g_k \}_{k \in \mathbb{I} }$ of vectors in a Hilbert space $\mathcal{H}$ is a Bessel sequence with Bessel bound $b_0$ if and only if the pre-frame operator $\mathcal{T}:\ell^2(\mathbb{I})\to \mathcal{H}$,
 	$\mathcal{T}(\{a_k\}_{k\in \mathbb{I}})=\sum\limits_{k\in \mathbb{I}}a_k g_k$
 	is bounded operator  from $\ell^2(\mathbb{I})$ into $\mathcal{H}$ with $\|\mathcal{T}\|^2\le b_0$.
 \end{thm}
 \begin{prop}\label{neww}\cite[p.\,107]{ole}
 	Let   $\{g_k \}_{k \in \mathbb{I} }$ be a sequence of vectors in $\mathcal{H}$, such that
 	 there exist a constant $b_0>0$ satisfying
 	\begin{align*}
 	\Big\|\sum a_k g_k\Big\|_{\mathcal{H}}^2\le b_0 \sum|a_k|^2,
 	\end{align*}
 	for all finite sequences $\{a_k\}$, then the series  $\sum_{k\in \mathbb{I}}a_k g_k$ converges for every $\{a_k \}_{k\in \mathbb{I}}\in \ell^2(\mathbb{I})$, further $\{g_k\}_{k\in \mathbb{I}}$ is a Bessel sequence with Bessel bound $b_0$.
 \end{prop}

  \begin{lem}\label{naya}\cite[p.\,94]{ole}
 	For any $b>0$ and any  constant $c_0\in \mathbb{R}$, $\{\sqrt{b}e^{2\pi i bl(\xi+c_0)}  \}_{l\in \mathbb{Z}}$ is a complete orthonormal system of $L^2(0,\frac{1}{b})$.
 \end{lem}

\subsection{Discrete Vector-valued Nonuniform Signal Spaces}
In this section, we give necessary background for structure of discrete vector-valued nonuniform signal spaces and some related basic results. As is standard,
symbols $\mathbb{N},\mathbb{Z},\mathbb{R}$ and $\mathbb{C}$ signify, the  set of natural numbers, integers, real numbers and complex numbers respectively. For $S\in \mathbb{N}$, the $S$-copies of $\mathbb{C}$ denoted as $\mathbb{C}^S$  is the complex euclidean space with norm $||\bold{X}||_{\mathbb{C}^S}= \sqrt{\sum\limits_{k=1}^S\big|[\bold{X}]_k\big|^2  }$ for ${\bold{X}}=\big(\big[{\bold{X}}\big]_1,\big[{\bold{X}}\big]_2,\dots, \big[{\bold{X}}\big]_S \big)\in \mathbb{C}^S,$ where $[{\bold{X}}]_j$ denotes the $j^th$ element of vector ${\bold{X}}$. The arithmetic mean of vector $\bold{X}$, defined  as $\text{\larger[3]$\mu$}_{\bold{X}}=\frac{1}{S}\sum\limits_{k=1}^S[\bold{X}]_k$, which is a scalar valued giving the average value of all the entries of the given vector. Let $N\in \mathbb{N}$  and $r$ be an odd integer coprime with $N$ such that $1\le r\le 2N-1$, assume $\Lambda=\{0,\frac{r}{N}\}+2\mathbb{Z}$. The \emph{discrete vector-valued nonuniform  signal  space}, denoted by $\ell^2(\Lambda,\mathbb{C}^S)$, is defined as
\begin{align*}
\ell^2(\Lambda,\mathbb{C}^S): =\Big\{\bold{Z}=\{\bold{Z}(\lambda)\}_{\lambda\in \Lambda}=\Bigg\{ \begin{bmatrix}
[\bold{Z}(\lambda)]_1
\vspace{5pt}\\
[\bold{Z}(\lambda)]_2\\
\vdots
\vspace{5pt}\\
[\bold{Z}(\lambda)]_S
\end{bmatrix}  \Bigg\}_{\lambda\in \Lambda}\subset \mathbb{C}^S :\sum\limits_{\substack{\lambda\in \Lambda\\k=1,2,\dots,S}}\big|\big[\bold{Z}(\lambda)\big]_k\big|^2 <\infty    \Big\}.
\end{align*}	
The space $\ell^2(\Lambda,\mathbb{C}^S)$ is a Hilbert space with respect to the inner product:
\begin{align*}
\langle \bold{Z},\bold{W}\rangle &=\sum\limits_{\substack{\lambda\in \Lambda\\k=1,2,\dots,S}}[\bold{Z}(\lambda)]_k\overline{[\bold{W}(\lambda)]_k}, \ \ \  \bold{Z},\, \bold{W}\in  \ell^2(\Lambda,\mathbb{C}^S).
\end{align*}
Its associated norm is given by
\begin{align*}
||\bold{Z}||_{\ell^2}&=\sqrt{\sum\limits_{\substack{\lambda\in \Lambda\\k=1,2,\dots,S}}\big|\big[\bold{Z}(\lambda)\big]_k\big|^2}, \ \  \bold{Z} \,\in  \ell^2(\Lambda,\mathbb{C}^S).
\end{align*}
Note for $N=1$ and $S=1$, this space becomes $\ell^2(\mathbb{Z},\mathbb{C})$, the standard  discrete signal space.  The arithmetic mean sequence of vector-valued nonuniform sequence $\bold{Z}\in \ell^2(\Lambda,\mathbb{C}^S)$ is denoted as $\text{\larger[3]$\mu$}_{\bold{Z}}=\{\text{\larger[3]$\mu$}_{\bold{Z}}(\lambda)\}_{\lambda\in \Lambda}$, where
\begin{align*}
\text{\larger[3]$\mu$}_{\bold{Z}}(\lambda)=\frac{1}{S}\sum\limits_{k=1}^S[\bold{Z}(\lambda)]_k,  \,\,  \lambda\in \Lambda.
\end{align*}
Clearly, $\text{\larger[3]$\mu$}_{\bold{Z}}\in \ell^2(\Lambda,\mathbb{C})$. The sum of vector-valued sequence $\bold{Z}\in \ell^2(\Lambda,\mathbb{C}^S)$ is defined as
\begin{align*}
\sum\limits_{\lambda\in \Lambda}\bold{Z}(\lambda)=\begin{bmatrix}
\sum\limits_{\lambda\in \Lambda}\big[\bold{Z}(\lambda)\big]_1
\vspace{5pt}\\
\sum\limits_{\lambda\in \Lambda}\big[\bold{Z}(\lambda)\big]_2\\
\vdots
\vspace{5pt}\\
\sum\limits_{\lambda\in \Lambda}\big[\bold{Z}(\lambda)\big]_S\\
\end{bmatrix},
\end{align*}
provided each  coordinate series converges in $\mathbb{C}$. Now we consider some important operators related to our work. Let $\lambda\in \Lambda, M\in \mathbb{N}$ be fixed, $m\in \{0,1,\dots,M-1\}$ and $\bold{Z}\in \ell^2(\Lambda,\mathbb{C}^S)$, we have
\begin{enumerate}
	\item \texttt{Shift operator} :  $R_{2N\lambda}: \ell^2(\Lambda,\mathbb{C}^S)\to \ell^2(\Lambda,\mathbb{C}^S)$,
	\begin{align*}
	R_{2N\lambda}\bold{Z}=\Bigg\{\begin{bmatrix}
	\big[\bold{Z}(\lambda'-2N\lambda)\big]_1
	\vspace{5pt}\\
	\big[\bold{Z}(\lambda'-2N\lambda)\big]_2\\
	\vdots
	\vspace{5pt}\\
	\big[\bold{Z}(\lambda'-2N\lambda)\big]_S
	\end{bmatrix}\Bigg\}_{\lambda'\in \Lambda}, \ \bold{Z} \in \ell^2(\Lambda,\mathbb{C}^S).
	\end{align*}
	\item \texttt{Modulation operator} :  $ E_{\frac{m}{M}}:\ell^2(\Lambda,\mathbb{C}^S)\to \ell^2(\Lambda,\mathbb{C}^S)$,
	\begin{align*}
	E_{\frac{m}{M}} \bold{Z}=\Bigg\{\begin{bmatrix}
	e^{2\pi i \frac{m}{M}\lambda'}\big[\bold{Z}(\lambda')\big]_1
	\vspace{5pt}\\
	e^{2\pi i \frac{m}{M}\lambda'}\big[\bold{Z}(\lambda')\big]_2\\
	\vdots
	\vspace{5pt}\\
	e^{2\pi i \frac{m}{M}\lambda'}\big[\bold{Z}(\lambda')\big]_S
	\end{bmatrix}\Bigg\}_{\lambda'\in \Lambda}.
	\end{align*}
	\end{enumerate}
For any Lebesgue measurable subset $G\subset \mathbb{R}$, $L^2(G)$ denotes the space of all square integrable functions on $G$, which forms a Hilbert space with inner product $\langle f_1,f_2\rangle=\int\limits_{G}f_1(\xi)\overline{f_2(\xi)}\,d\xi,\,\, f_1,f_2\in L^2(G)$.

Gabardo  and Nashed \cite{GN1} considered  the following definition in the study of nonuniform  wavelets in the space $L^2(\mathbb{R})$.
\begin{defn} \cite{GN1}\label{SPDefn}
	Let $\Omega\subset\mathbb{R}$ be measurable and $\Lambda \subset \mathbb{R}$ a countable subset. If the collection $\{|\Omega|^{-\frac{1}{2}}e^{2\pi i \lambda\cdot} \chi_{\Omega}(\cdot) \}_{\lambda\in \Lambda}$ forms complete orthonormal system for  $L^2(\Omega)$, where $\chi_{\Omega}$ is indicator function on $\Omega$ and $|\Omega|$ is Lebesgue measure of $\Omega$, then the pair $(\Omega,\Lambda)$ is a  \emph{spectral pair}.
	\end{defn}
\begin{exa}\cite{GN1}\label{ex11}
	Let  $N \in \mathbb{N}$, $r $ be an odd integer  coprime to $N$ such that $1\le r \le2N-1$, and let   $\Lambda =\left\{0,\frac{r}{N}  \right\} +2\mathbb{Z} $ and $\Omega=[0,\frac{1}{2}[ \cup [\frac{N}{2},\frac{N+1}{2}[ $. Then,  $(\Omega,\Lambda)$ is a spectral pair.
\end{exa}

 Now, we define the corresponding vector-valued nonuniform function space  of a given discrete vector-valued nonuniform  signal  space $\ell^2(\Lambda,\mathbb{C}^S)$, the Fourier transformation  and   its  elementary properties.
  Let $\Lambda=\{0,\frac{r}{N}\}+2\mathbb{Z}$, where $N\in \mathbb{N}$, $r$ is an odd integer coprime with $N$ satisfying $1\le r\le 2N-1$  and $\Omega=[0,\frac{1}{2}[\cup[\frac{N}{2},\frac{N+1}{2}[$. The vector-valued nonuniform function space $L^2(\Omega,\mathbb{C}^S)$ is defined as follows:
\begin{align*}
L^2(\Omega,\mathbb{C}^S)=\Big\{\bold{F}=\bold{F}(\xi)= \begin{bmatrix}
\big[\bold{F}(\xi)\big]_1
\vspace{5pt}\\
\big[\bold{F}(\xi)\big]_2\\
\vdots
\vspace{5pt}\\
\big[\bold{F}(\xi)\big]_S
\end{bmatrix}:[\bold{F}(\xi)]_k:\Omega\to\mathbb{C}, \int\limits_{\Omega} \big|\big[\bold{F}(\xi)\big]_k\big|^2 <\infty,1\le k \le S  \Big\}.
\end{align*}
The space $L^2(\Omega,\mathbb{C}^S)$  is  a Hilbert space with respect to the inner product defined by
\begin{align*}
\langle \bold{F},\bold{G} \rangle&= \sum\limits_{k=1}^{S}\int\limits_{\Omega}\big[\bold{F}(\xi)\big]_k\overline{\big[\bold{G}(\xi)\big]_k}\,d\xi, \ \ \bold{F}, \bold{G} \in L^2(\Omega,\mathbb{C}^S).
\end{align*}
This inner product induces the following norm:
\begin{align*}
||\bold{F}||_{L^2}
&=\sqrt{\sum\limits_{k=1}^{S}\int\limits_{\Omega}\Big|\big[\bold{F}(\xi)\big]_k\Big|^2}, \ \bold{F} \in L^2(\Omega,\mathbb{C}^S).
\end{align*}
 The map  $\mathcal{F}:\ell^2(\Lambda,\mathbb{C}^S)\to L^2(\Omega,\mathbb{C}^S)$ defined  by
\begin{align*}
\mathcal{F}(\bold{Z})=\begin{bmatrix}
\sum\limits_{\lambda\in \Lambda}\big[\bold{Z}(\lambda)\big]_1\,e^{2\pi i \lambda\xi}
\vspace{5pt}\\
\sum\limits_{\lambda\in \Lambda}\big[\bold{Z}(\lambda)\big]_2\,e^{2\pi i \lambda\xi}\\
\vdots
\vspace{5pt}\\
\sum\limits_{\lambda\in \Lambda}\big[\bold{Z}(\lambda)\big]_S\, e^{2\pi i \lambda\xi}
\end{bmatrix}
\end{align*}
is called the \emph{Fourier transform} of $\bold{Z} \in \ell^2(\Lambda,\mathbb{C}^S)$. The Fourier transform $\mathcal{F}$   is well defined, since $(\Lambda,\Omega)$ is a spectral pair by Example \ref{ex11}, so each coordinate series will converge in $L^2(\Omega,\mathbb{C})$. Further $\mathcal{F}$ is bijective and its inverse, that is, inverse Fourier transform   $\mathcal{F}^{-1}:L^2(\Omega,\mathbb{C}^S)\to \ell^2(\Lambda,\mathbb{C}^S) $ is given by
\begin{align*}
\mathcal{F}^{-1}(\bold{F})=\begin{bmatrix}
\int\limits_{\Omega}\big[\bold{F}(\xi)\big]_1\,e^{-2\pi i \lambda\xi}\,d\xi
\vspace{5pt}\\
\int\limits_{\Omega}\big[\bold{F}(\xi)\big]_2\,e^{-2\pi i \lambda\xi}\,d\xi\\
\vdots
\vspace{5pt}\\
\int\limits_{\Omega}\big[\bold{F}(\xi)\big]_S\,e^{-2\pi i \lambda\xi}\,d\xi
\end{bmatrix}.
\end{align*}
\begin{rem}
	It can be easily observed that $2N \Lambda\subset2\mathbb{Z}\subset\Lambda$ and $\Lambda+2N\lambda=\Lambda$ for any $\lambda\in \Lambda.$
\end{rem}
Now we give some identities related to the Fourier transform. Let  $\bold{Z},\bold{W}\in \ell^2(\Lambda,\mathbb{C}^S)$ and $\lambda\in \Lambda$. Then,
\begin{enumerate}
	\item $\mathcal{F}(R_{2N\lambda} \bold{Z})(\xi)=e^{4\pi i N\lambda\xi}\mathcal{F}(\bold{Z})(\xi).   $
	\item $\langle \bold{Z}, \bold{W}\rangle_{\ell^2} = \langle \mathcal{F}(\bold{Z}), \mathcal{F}(\bold{W})\rangle_{L^2}  $ \ \quad (\emph{Parseval's relation}).
	\item $||\bold{Z}||_{\ell^2}=||\mathcal{F}(\bold{Z})||_{L^2}$ \ \quad (\emph{Plancherel's formula}).
	\item $\mathcal{F}\big(E_{\frac{m}{M}}R_{2N\lambda}\bold{W}\big)(\xi)=e^{4\pi iN\lambda(\frac{m}{M}+\xi) } \mathcal{F}\big(E_{\frac{m}{M}}\bold{W}\big)(\xi)  $.
\end{enumerate}
 We conclude this section with the following lemma; which can be proved by mathematical induction method.
 \begin{lem}\label{lemmahari}
For any  $z_1,z_2,\dots,z_t\in \mathbb{C}$, we have
 	$\Big|\sum\limits_{k=1}^{t}z_k\Big|^2\le2^{t-1}\sum\limits_{k=1}^{t}|z_k|^2$.
\end{lem}

\section{Frames of Discrete  Vector-valued     Nonuniform  Gabor  Systems}\label{Sect3}
This section deals with discrete vector-valued nonuniform Gabor frames in discrete vector-valued nonuniform signal spaces.
We begin with the  definition of a  discrete  vector-valued nonuniform  Gabor system in $\ell^2(\Lambda,\mathbb{C}^S)$.
\begin{defn}\label{DDEFNI}
Let $N,M,S \in \mathbb{N},P\in \mathbb{N}\cup \{0\} $ and  $r $ be an odd integer coprime with $N$ such that $1\le r \le 2N-1$. Let  $\Lambda =\left\{0,\frac{r}{N}  \right\} +2\mathbb{Z} $, $T_1=\{0,1,\dots,M-1\},T_2=\{0,1,\dots,P\}$ and $\{\bold{W}_j\}_{j\in T_2}\subset \ell^2(\Lambda,\mathbb{C}^S)$. A collection  of sequences of the form
	\begin{align}\label{system}
	\mathcal{G}(\Lambda,  T_1, T_2, \bold{W}_j) :=\big\{E_{\frac{m}{M}}R_{2N\lambda}\bold{W}_j :\lambda\in \Lambda,m\in T_1,j\in T_2\big\}
	\end{align}
is called a \emph{discrete  vector-valued nonuniform  Gabor system}  (DVNUG system,  in short) for discrete vector-valued nonuniform signal  space $\ell^2(\Lambda,\mathbb{C}^S)$;  and $\bold{W}_j$ $( j\in T_2)$  are called      \emph{ window sequences}.
\end{defn}

\begin{defn}\label{DefnnonI}
The collection of  sequences $\mathcal{G}(\Lambda,  T_1, T_2, \bold{W}_j)$, as   defined in \eqref{system}, is called a  \emph{discrete vector valued nonuniform  Gabor frame} (DVNUG frame, in short) for $\ell^2(\Lambda,\mathbb{C}^S)$, if there exist constants $0 < \alpha \leq \beta < \infty$  such that
\begin{align}\label{bessel}
\alpha ||\bold{Z}||_{\ell^2}^2\le \sum\limits_{\substack{ m\in T_1,j\in T_2\\  \lambda\in \Lambda}}\Big|\langle \bold{Z},E_{\frac{m}{M}}R_{2N\lambda}\bold{W}_j\rangle\Big|^2\le \beta ||\bold{Z}||_{\ell^2}^2\,\,\text{for all}\,\, \bold{Z}\in \ell^2(\Lambda,\mathbb{C}^S).
\end{align}
\end{defn}
The scalars $\alpha$ and $\beta$, respectively, known as  l\emph{ower frame bound} and \emph{upper frame bound}.  If only the right-hand side inequality in \eqref{bessel} holds, then we say       that $\mathcal{G}(\Lambda,  T_1, T_2, \bold{W}_j)$        is   a  \emph{DVNUG Bessel sequence}    with Bessel bound $\beta$.

Before giving the  operators associated with DVNUG frames for discrete vector-valued nonuniform signal space, we consider the  space
\begin{align*}
\ell^2(\Lambda,T_1,T_2,\mathbb{C}) :=\Big\{ \{a_{\lambda,m,j}\}_{\substack{ m\in T_1,j\in T_2\\  \lambda\in \Lambda}}\subset \mathbb{C}:\sum\limits_{\substack{ m\in T_1,j\in T_2\\  \lambda\in \Lambda}}|a_{\lambda,m,j}|^2<\infty  \Big\}.
\end{align*}
This space is a Hilbert space with respect to  the  following inner product.
\begin{align}\label{eqIP}
\Big\langle \{a_{\lambda,m,j}\}_{\substack{ m\in T_1,j\in T_2\\  \lambda\in \Lambda}}, \{b_{\lambda,m,j}\}_{\substack{ m\in T_1,j\in T_2\\  \lambda\in \Lambda}} \Big\rangle&=\sum\limits_{\substack{ m\in T_1,j\in T_2\\  \lambda\in \Lambda}}a_{\lambda,m,j}\overline{b_{\lambda,m,j}},
\end{align}
for $\{a_{\lambda,m,j}\}_{\substack{ m\in T_1,j\in T_2\\  \lambda\in \Lambda}}$, $\{b_{\lambda,m,j}\}_{\substack{ m\in T_1,j\in T_2\\  \lambda\in \Lambda}} \in \ell^2(\Lambda,T_1,T_2,\mathbb{C})$.
The norm associated with the  inner product given in \eqref{eqIP} is given by the following formula:
\begin{align*}
||\{a_{\lambda,m,j}\}_{\substack{ m\in T_1,j\in T_2\\  \lambda\in \Lambda}}||_{\ell^2}&=\sqrt{\sum\limits_{\substack{ m\in T_1,j\in T_2\\  \lambda\in \Lambda}}|a_{\lambda,m,j}|^2},  \ \ \  \{a_{\lambda,m,j}\}_{\substack{ m\in T_1,j\in T_2\\  \lambda\in \Lambda}} \in \ell^2(\Lambda,T_1,T_2,\mathbb{C}).
\end{align*}

If $\mathcal{G}(\Lambda,  T_1, T_2, \bold{W}_j)$  is  a  discrete vector valued nonuniform Bessel sequence    with Bessel bound $\beta$, then the  map $\mathcal{T}:\ell^2(\Lambda,T_1,T_2,\mathbb{C})\to \ell^2(\Lambda,\mathbb{C}^S)$  defined by
\begin{align*}
\mathcal{T}\Big(\{ a_{\lambda,m,j}\}_{\substack{ m\in T_1,j\in T_2\\  \lambda\in \Lambda}}\Big)=\sum\limits_{\substack{ m\in T_1,j\in T_2\\  \lambda\in \Lambda}}a_{\lambda,m,j}E_{\frac{m}{M}}R_{2N\lambda}\bold{W}_j,  \ \  \{ a_{\lambda,m,j}\}_{\lambda\in \Lambda, \, j\in S} \in \ell^2(\Lambda,T_1,T_2,\mathbb{C})
\end{align*}
is called the \emph{pre-frame operator} of $\mathcal{G}(\Lambda,  T_1, T_2, \bold{W}_j)$ . The pre-frame operator $\mathcal{T}$ is linear and bounded  with $||\mathcal{T}||\le \sqrt{\beta}$. The Hilbert-adjoint operator of $\mathcal{T}$ is  $\mathcal{T}^*:\ell^2(\Lambda,\mathbb{C}^S)\to \ell^2(\Lambda,T_1,T_2,\mathbb{C})$ is called the \emph{analysis operator},  is given by $\mathcal{T}^*(\bold{Z}) = \{\langle \bold{Z}, E_{\frac{m}{M}}R_{2N\lambda}\bold{W}_j\rangle\}_{\substack{ m\in T_1,j\in T_2\\  \lambda\in \Lambda}}$, $\bold{Z} \in \ell^2(\Lambda,\mathbb{C}^S)$. The composition $\Xi=\mathcal{T}\mathcal{T}^*:\ell^2(\Lambda,\mathbb{C}^S)\to \ell^2(\Lambda,\mathbb{C}^S)$  given by
\begin{align*}
\Xi(\bold{Z})=\sum\limits_{\substack{ m\in T_1,j\in T_2\\  \lambda\in \Lambda}}\langle \bold{Z},E_{\frac{m}{M}}R_{2N\lambda}\bold{W}_j\rangle E_{\frac{m}{M}}R_{2N\lambda}\bold{W}_j,\,\,\bold{Z}\in \ell^2(\Lambda,\mathbb{C}^S),
\end{align*}
is called  the \emph{frame operator }of $\mathcal{G}(\Lambda,  T_1, T_2, \bold{W}_j)$.  The frame operator $\Xi$ is bounded and linear; if  $\mathcal{G}(\Lambda,  T_1, T_2, \bold{W}_j)$  is  a DVNUG frame for $\ell^2(\Lambda,\mathbb{C}^S)$,  then  $\Xi$ is invertible on $\ell^2(\Lambda,\mathbb{C}^S)$.

Throughout the entire paper, unless otherwise stated, we assume  $N$, $M$, $S \in \mathbb{N}$, $P\in \mathbb{N}\cup \{0\} $ and  $r $ be an odd integer coprime with $N$ such that $1\le r \le 2N-1$,  $\Lambda =\left\{0,\frac{r}{N}  \right\} +2\mathbb{Z},\, \Omega=[0,\frac{1}{2}[\cup [\frac{N}{2},\frac{N+1}{2}[$, $T_1=\{0,1,\dots,M-1\}$, $T_2=\{0,1,\dots,P\}$ and $\mathcal{G}(\Lambda,  T_1, T_2, \bold{W}_j)=\big\{E_{\frac{m}{M}}R_{2N\lambda}\bold{W}_j :\lambda\in \Lambda,m\in T_1,j\in T_2\big\}$, where $\bold{W}_j\in \ell^2(\Lambda,\mathbb{C}^S)$.

\vspace{8pt}

The following result  provides a  sufficient condition for the existence of   DVNUG Bessel sequences in discrete vector-valued nonuniform signal spaces,  with  explicit Bessel bounds,  in terms of Fourier transformations of the modulated window sequences.
\begin{thm}\label{th1}
	Let $\{\bold{W}_j\}_{j\in T_2} \subset \ell^2(\Lambda,\mathbb{C}^S)$. For   $m\in T_1$,  $j\in T_2$  and  a.e. $\xi\in \Omega$, let
\begin{align*}
\Big|\Big|\mathcal{F}(E_{\frac{m}{M}}\bold{W}_j)(\xi)\Big|\Big|_{\mathbb{C}^S}\le B_0<\infty,
\end{align*}
for some positive constant $B_0$. Then, $\mathcal{G}(\Lambda,  T_1, T_2, \bold{W}_j)$  is  a DVNUG Bessel sequence with Bessel bound $2^{(M+P)} B_0^2S$.
\end{thm}
\begin{proof}
 Let  $\{a_{\lambda,m,j}\}$  be any finite sequence  in $\ell^2(\Lambda,T_1,T_2,\mathbb{C})$. Then, we have
	
	\begin{align}\label{a}
	\Big|\Big|\sum\limits_{\substack{ m\in T_1,j\in T_2\\  \lambda\in \Lambda}}a_{\lambda,m,j}E_{\frac{m}{M}}R_{2N\lambda}\bold{W}_j\Big|\Big|_{\ell^2}&=\Big|\Big|\mathcal{F}\Big(\sum\limits_{\substack{ m\in T_1,j\in T_2\\  \lambda\in \Lambda}}a_{\lambda,m,j}E_{\frac{m}{M}}R_{2N\lambda}\bold{W}_j\Big)\Big|\Big|_{L^2}\nonumber\\
	&\le \sum\limits_{\substack{ m\in T_1\\j\in T_2}}\Big|\Big|\sum\limits_{\lambda\in \Lambda}a_{\lambda,m,j}e^{4\pi i N\lambda(\frac{m}{M}+\xi)}\mathcal{F}(E_{\frac{m}{M}}\bold{W}_j)(\xi) \Big|\Big|_{L^2}\nonumber\\
	&=\sum\limits_{\substack{ m\in T_1\\j\in T_2}}\sqrt{\sum\limits_{k=1}^S\int\limits_{\Omega}\Big|\sum\limits_{\lambda\in \Lambda}a_{\lambda,m,j}e^{4\pi i N\lambda(\frac{m}{M}+\xi)}\Big[\mathcal{F}(E_{\frac{m}{M}}\bold{W}_j)(\xi)\Big]_k\Big|^2 d\,\xi}\nonumber\\
		&=B_0 \sqrt{S}\sum\limits_{\substack{ m\in T_1\\j\in T_2}}\sqrt{\int\limits_{\Omega}\Big|\sum\limits_{\lambda\in \Lambda}a_{\lambda,m,j}e^{4\pi i N\lambda(\frac{m}{M}+\xi)}\Big|^2 d\,\xi  }.
	\end{align}
Using \eqref{a} and the fact $\{e^{4\pi i N\lambda(\frac{m}{M}+\xi)} \}_{\lambda\in\Lambda}$ is  orthonormal system in $L^2(\Omega)$, we have
	\begin{align}\label{b}
\Big|\Big|\sum\limits_{\substack{ m\in T_1,j\in T_2\\  \lambda\in \Lambda}}a_{\lambda,m,j}E_{\frac{m}{M}}R_{2N\lambda}\bold{W}_j\Big|\Big|_{\ell^2}  &\le B_0\sqrt{S} \sum\limits_{\substack{ m\in T_1\\j\in T_2}}\sqrt{\sum\limits_{\lambda\in \Lambda} |a_{\lambda,m,j}|^2 }.
	\end{align}
	Therefore by \eqref{b} and Lemma \ref{lemmahari}, we get
	\begin{align}
		\label{newq}
	\Big|\Big|\sum\limits_{\substack{ m\in T_1,j\in T_2\\  \lambda\in \Lambda}}a_{\lambda,m,j}E_{\frac{m}{M}}R_{2N\lambda}\bold{W}_j\Big|\Big|^2_{\ell^2} & \le \Bigg( B_0\sqrt{S} \sum\limits_{\substack{ m\in T_1\\j\in T_2}}\sqrt{\sum\limits_{\lambda\in \Lambda} |a_{\lambda,m,j}|^2 }\Bigg)^2 \nonumber \\
		                                                                                                                                                                                          & \le 2^{(M+P)}B_0^2 S\sum\limits_{\substack{ m\in T_1\\j\in T_2}}\sum\limits_{\lambda\in \Lambda}|a_{\lambda,m,j}|^2.
	\end{align}
By invoking Proposition \ref{neww} 	 and   Ineq. \eqref{newq}, we conclude that $\mathcal{G}(\Lambda,  T_1, T_2, \bold{W}_j)$ is a Bessel sequence with the desired Bessel bound.
	\end{proof}
The following example illustrates Theorem \ref{th1}.
\begin{exa}\label{bol}
	Let $ N=2$, $r=1$, $S=2$, $M=2$ and $P=7$. Then,  $\Lambda=\{0,\frac{1}{2}\}+2\mathbb{Z}$,  $\Omega=[0,\frac{1}{2}[ \cup [1,\frac{3}{2}[$, $T_1=\{0,1\}$ and $T_2=\{ 0,1,\dots,7\}$. For $j\in T_2$, define
$\bold{W}_j=\{\bold{W}_j(\lambda) \}_{\lambda\in \Lambda}\in \ell^2(\Lambda,\mathbb{C}^2) $ as follows:
\begin{align*}
&\bold{W}_0(0)=\begin{bmatrix}
1\\
0
\end{bmatrix}, \,\, \,\, \bold{W}_0(4)=\begin{bmatrix}
	1\\
	0
	\end{bmatrix}  ,\,\,\,\, \bold{W}_0(\lambda)=\begin{bmatrix}
	0\\
	0
	\end{bmatrix} \,\text{for}\,\lambda\in \Lambda\setminus\{0,4\};\\
	&\bold{W}_1(0)=\begin{bmatrix}
	1\\
	0
	\end{bmatrix}, \,\, \,\, \bold{W}_1(4)=\begin{bmatrix}
	-1\\
	0
	\end{bmatrix},  \,\,\,\, \bold{W}_1(\lambda)=\begin{bmatrix}
	0\\
	0
	\end{bmatrix} \,\text{for}\,\lambda\in \Lambda\setminus\{0,4\};\\
	&\bold{W}_2(0)=\begin{bmatrix}
	0\\
	1
	\end{bmatrix}, \,\, \,\,  \bold{W}_2(4)=\begin{bmatrix}
	0\\
	1
	\end{bmatrix}  ,\,\,\,\, \bold{W}_2(\lambda)=\begin{bmatrix}
	0\\
	0
	\end{bmatrix} \,\text{for}\,\lambda\in \Lambda\setminus\{0,4\};\\
	&\bold{W}_3(0)=\begin{bmatrix}
	0\\
	1
	\end{bmatrix}, \,\, \,\, \bold{W}_3(4)=\begin{bmatrix}
	0\\
	-1
	\end{bmatrix}  ,\,\,\,\, \bold{W}_3(\lambda)=\,\begin{bmatrix}
	0\\
	0
	\end{bmatrix} \,\text{for}\,\lambda\in \Lambda\setminus\{0,4\};\\
	&\bold{W}_4\Big(\frac{1}{2}\Big)=\begin{bmatrix}
	1\\
	0
	\end{bmatrix}, \,\, \,\, \bold{W}_4\Big(\frac{1}{2}+4\Big)=\begin{bmatrix}
	1\\
	0
	\end{bmatrix}, \,\, \,\, \bold{W}_4(\lambda)=\begin{bmatrix}
	0\\
	0
	\end{bmatrix} \,\text{for}\,\lambda\in \Lambda\setminus\Big\{\frac{1}{2},\frac{1}{2}+4\Big\};\\
	&\bold{W}_5\Big(\frac{1}{2}\Big)=\begin{bmatrix}
	1\\
	0
	\end{bmatrix}, \bold{W}_5\Big(\frac{1}{2}+4\Big)=\begin{bmatrix}
	-1\\
	0
	\end{bmatrix}, \,\,\,\, \bold{W}_5(\lambda)=\begin{bmatrix}
	0\\
	0
	\end{bmatrix} \,\text{for}\,\lambda\in \Lambda\setminus\Big\{\frac{1}{2},\frac{1}{2}+4\Big\};\\
	&\bold{W}_6\Big(\frac{1}{2}\Big)=\begin{bmatrix}
	0\\
	1
	\end{bmatrix}, \,\, \,\, \bold{W}_6\Big(\frac{1}{2}+4\Big)=\begin{bmatrix}
	0\\
	1
	\end{bmatrix},\,\, \,\,  \bold{W}_6(\lambda)=\begin{bmatrix}
	0\\
	0
	\end{bmatrix}\,\text{for}\,\lambda\in \Lambda\setminus\Big\{\frac{1}{2},\frac{1}{2}+4\Big\};\\
	&\bold{W}_7\Big(\frac{1}{2}\Big)=\begin{bmatrix}
	0\\
	1
	\end{bmatrix}, \,\, \,\, \bold{W}_7\Big(\frac{1}{2}+4\Big)=\begin{bmatrix}
	0\\
	-1
	\end{bmatrix},  \, \,\, \bold{W}_7(\lambda)=\begin{bmatrix}
	0\\
	0
	\end{bmatrix} \,\text{for}\,\lambda\in \Lambda\setminus\Big\{\frac{1}{2},\frac{1}{2}+4\Big\}.
	\end{align*}	
For $m\in T_1=\{0,1 \},j\in T_2  $, the system  $E_{\frac{m}{2}}\bold{W}_j=\big\{(E_{\frac{m}{2}}\bold{W}_j)(\lambda)\big\}_{\lambda\in \Lambda} $
	are as follows:
	\begin{align*}
	&\big(E_{\frac{m}{2}} \bold{W}_0\big)(0)=\begin{bmatrix}
	1\\
	0
	\end{bmatrix},\, \,\,\, \big(E_{\frac{m}{2}} \bold{W}_0\big)(4)=\begin{bmatrix}
	1\\
	0
	\end{bmatrix}, \,\,\,\,  \big(E_{\frac{m}{2}} \bold{W}_0\big)(\lambda)=\begin{bmatrix}
	0\\
	0
	\end{bmatrix} \,\text{for}\,\lambda\in \Lambda\setminus\{0,4\};\\
	&\big(E_{\frac{m}{2}} \bold{W}_1\big)(0)=\begin{bmatrix}
	1\\
	0
	\end{bmatrix},\, \,\,\,  \big(E_{\frac{m}{2}} \bold{W}_1\big)(4)=\begin{bmatrix}
	-1\\
	0
	\end{bmatrix}, \,\,\,\, \big(E_{\frac{m}{2}} \bold{W}_1\big)(\lambda)=\begin{bmatrix}
	0\\
	0
	\end{bmatrix} \,\text{for}\,\lambda\in \Lambda\setminus\{0,4\};\\
	&\big(E_{\frac{m}{2}} \bold{W}_2\big)(0)=\begin{bmatrix}
	0\\
	1
	\end{bmatrix},\, \,\,\, \big(E_{\frac{m}{2}} \bold{W}_2\big)(4)\,\,=\,\,\begin{bmatrix}
	0\\
	1
	\end{bmatrix}  ,\,\, \,\,  \big(E_{\frac{m}{2}} \bold{W}_2\big)(\lambda)=\begin{bmatrix}
	0\\
	0
	\end{bmatrix} \,\text{for}\,\lambda\in \Lambda\setminus\{0,4\};\\
	&\big(E_{\frac{m}{2}} \bold{W}_3\big)(0)=\begin{bmatrix}
	0\\
	1
	\end{bmatrix},\, \,\,\,\,   \big(E_{\frac{m}{2}} \bold{W}_3\big)(4)=\begin{bmatrix}
	0\\
	-1
	\end{bmatrix}  ,\,\, \,\, \big(E_{\frac{m}{2}} \bold{W}_3\big)(\lambda)=\begin{bmatrix}
	0\\
	0
	\end{bmatrix}\,\text{for}\,\lambda\in \Lambda\setminus\{0,4\}; \\
	&\big(E_{\frac{m}{2}} \bold{W}_4\big)\Big(\frac{1}{2}\Big)=\begin{bmatrix}
	e^{\pi i \frac{m}{2}}\\
	0
	\end{bmatrix},  \,\,  \big(E_{\frac{m}{2}} \bold{W}_4\big)\Big(\frac{1}{2}+4\Big)=\begin{bmatrix}
	e^{\pi i \frac{m}{2}}\\
	0
	\end{bmatrix}, \,\, \,\,  \big(E_{\frac{m}{2}} \bold{W}_4\big)(\lambda)=\begin{bmatrix}
	0\\
	0
	\end{bmatrix} \text{for}\,\lambda\in \Lambda\setminus\Big\{\frac{1}{2},\frac{1}{2}+4\Big\};\\
	&\big(E_{\frac{m}{2}} \bold{W}_5\big)\Big(\frac{1}{2}\Big)=\begin{bmatrix}
	e^{\pi i \frac{m}{2}}\\
	0
	\end{bmatrix}, \big(E_{\frac{m}{2}} \bold{W}_5\big)\Big(\frac{1}{2}+4\Big)=\begin{bmatrix}
	-e^{\pi i \frac{m}{2}}\\
	0
	\end{bmatrix}, \,\, \,\, \big(E_{\frac{m}{2}} \bold{W}_5\big)(\lambda)=\begin{bmatrix}
	0\\
	0
	\end{bmatrix} \text{for}\,\lambda\in \Lambda\setminus\Big\{\frac{1}{2},\frac{1}{2}+4\Big\};\\
	&\big(E_{\frac{m}{2}} \bold{W}_6\big)\Big(\frac{1}{2}\Big)=\begin{bmatrix}
	0\\
	e^{\pi i \frac{m}{2}}
	\end{bmatrix},\big(E_{\frac{m}{2}} \bold{W}_6\big)\Big(\frac{1}{2}+4\Big)=\begin{bmatrix}
	0\\
	e^{\pi i \frac{m}{2}}
	\end{bmatrix},  \,\, \,\,  \big(E_{\frac{m}{2}} \bold{W}_6\big)(\lambda)=\begin{bmatrix}
	0\\
	0
	\end{bmatrix} \text{for}\,\lambda\in \Lambda\setminus\Big\{\frac{1}{2},\frac{1}{2}+4\Big\};\\
	&\big(E_{\frac{m}{2}} \bold{W}_7\big)\Big(\frac{1}{2}\Big)=\begin{bmatrix}
	0\\
	e^{\pi i \frac{m}{2}}
	\end{bmatrix},\big(E_{\frac{m}{2}} \bold{W}_7\big)\Big(\frac{1}{2}+4\Big)=\begin{bmatrix}
	0\\
	-e^{\pi i \frac{m}{2}}
	\end{bmatrix},\,\, \,\, \big(E_{\frac{m}{2}} \bold{W}_7\big)(\lambda)=\begin{bmatrix}
	0\\
	0
	\end{bmatrix} \,\text{for}\,\lambda\in \Lambda\setminus\Big\{\frac{1}{2},\frac{1}{2}+4\Big\}.
	\end{align*}	
	Now for $m\in T_1$ and $j\in T_2$, the corresponding Fourier transforms $\mathcal{F}(E_{\frac{m}{M}}\bold{W}_j)(\xi)$
	are given by:	
	\begin{align*}
	&\mathcal{F}(E_{\frac{m}{M}}\bold{W}_0)(\xi)=\begin{bmatrix}
	1+e^{8\pi i \xi}\\
	0
	\end{bmatrix},\,\,\,\,\,\,\,\,\,\,\quad\quad\,\,\,\,\quad \mathcal{F}(E_{\frac{m}{M}}\bold{W}_1)(\xi)=\begin{bmatrix}
	1-e^{8\pi i \xi}\\
	0
	\end{bmatrix},\\ &\mathcal{F}(E_{\frac{m}{M}}\bold{W}_2)(\xi)=\begin{bmatrix}
	0\\
	1+e^{8\pi i \xi}
	\end{bmatrix},
	\,\,\,\,\,\,\,\,\,\,\quad\quad\,\,\,\,\quad\mathcal{F}(E_{\frac{m}{M}}\bold{W}_3)(\xi)=\begin{bmatrix}
	0\\
	1-e^{8\pi i \xi}
	\end{bmatrix},\\ &\mathcal{F}(E_{\frac{m}{M}}\bold{W}_4)(\xi)=\begin{bmatrix}
	e^{\pi i\frac{m}{2}}(e^{\pi i \xi}+e^{9\pi i\xi})\\
	0
	\end{bmatrix}, \,\,\,\,\,\mathcal{F}(E_{\frac{m}{M}}\bold{W}_5)(\xi)=\begin{bmatrix}
	e^{\pi i\frac{m}{2}}(e^{\pi i \xi}-e^{9\pi i\xi})\\
	0
	\end{bmatrix},\\
	&
	\mathcal{F}(E_{\frac{m}{M}}\bold{W}_6)(\xi)=\begin{bmatrix}
	0\\
	e^{\pi i\frac{m}{2}}(e^{\pi i \xi}+e^{9\pi i\xi})
	\end{bmatrix}, \ \text{and} \ \,\,\,\mathcal{F}(E_{\frac{m}{M}}\bold{W}_7)(\xi)=\begin{bmatrix}
	0\\
	e^{\pi i\frac{m}{2}}(e^{\pi i \xi}-e^{9\pi i\xi})
	\end{bmatrix}.	
	\end{align*}
By using Lemma \ref{lemmahari}, it is easy to calculate  for $m\in T_1, j\in T_2$ that,
\begin{align*}
\Big|\Big|\mathcal{F}(E_{\frac{m}{M}}\bold{W}_j)(\xi)\Big|\Big|_{\mathbb{C}^S}\le 2.
\end{align*}
Thus,  by Theorem \ref{th1}, the DVNUG system  $\mathcal{G}(\Lambda,  T_1, T_2, \bold{W}_j)=\big \{E_{\frac{m}{2}}R_{4\lambda}\bold{W}_j:m\in T_1,j\in T_2,\lambda \in \{0,\frac{1}{2}\}+2\mathbb{Z}  \big \}$ is a Bessel sequence with Bessel  bound $2^{12}$.
\end{exa}

The next result  gives  a  necessary conditions for DVNUG Bessel sequence in discrete vector-valued nonuniform signal spaces.
\begin{thm}\label{th2}
Let $\{\bold{W}_j\}_{j\in T_2} \subset \ell^2(\Lambda,\mathbb{C}^S)$ be such that  $\mathcal{G}(\Lambda,  T_1, T_2, \bold{W}_j)$ is a DVNUG Bessel sequence with Bessel bound $B_0$ for $\ell^2(\Lambda,\mathbb{C}^S)$.  Then,  for $m\in T_1$, $j\in T_2$ and $a.e.\,\xi \in \Omega$, we have,
	\begin{align*}
	\Big|\Big|\mathcal{F}(E_{\frac{m}{M}}\bold{W}_j)(\xi)\Big|\Big|_{\mathbb{C}^S}\le 2\sqrt{NB_0}.
	\end{align*}
	
\end{thm}
\begin{proof}
	Let $m_0\in T_1$, $j_0\in T_2$ be  arbitrarily fixed elements. Let $f(\xi)$ be an any element of $L^2(0,\frac{1}{4N})$. Then,  by Lemma \ref{naya},  $\{\sqrt{4N}e^{2\pi i(4N)l(\xi+\frac{m_0}{M})}\}_{l\in \mathbb{Z}}$ is a complete orthonormal system of $L^2(0,\frac{1}{4N})$, so there exists sequence $c_l$ such that
	\begin{align}\label{aa123}
	f(\xi)=\sum\limits_{l\in \mathbb{Z}}c_l\sqrt{4N}e^{2\pi i(4N)l(\xi+\frac{m_0}{M})},
	\end{align}
	satisfying $\int\limits_{0}^{\frac{1}{4N}}|f(\xi)|^2\, d\xi=\sum\limits_{l\in\mathbb{Z}}|c_l|^2<\infty$.

	Define sequence $\{a_{\lambda,m,j}\}_{{\substack{ m\in T_1,j\in T_2\\  \lambda\in \Lambda}}}$ as follows:
	
	\begin{align}\label{c}
	a_{\lambda,m,j}=
	\begin{cases}
	c_{l}, & \text{if} \,\, \lambda=2l, \ l \in \mathbb{Z}, m=m_0,j=j_0;\,\\ 	
	0,& \text{otherwise}.
	\end{cases}
	\end{align}
	Obviously $\{a_{\lambda,m,j}\}_{{\substack{ m\in T_1,j\in T_2\\  \lambda\in \Lambda}}} \in \ell^2(\Lambda,T_1,T_2,\mathbb{C})$.
Since,  $\mathcal{G}(\Lambda,  T_1, T_2, \bold{W}_j)$  is a Bessel sequence with Bessel bound $B_0$, so by Theorem \ref{nh1},  the associated pre-frame operator $\mathcal{T}:\ell^2(\Lambda,T_1,T_2,\mathbb{C})\to \ell^2(\Lambda,\mathbb{C}^S)   $ is bounded and satisfies
\begin{align}\label{d}
 ||T\{a_{\lambda,m,j}\}||_{\ell^2}^2\le B_0\sum\limits_{\substack{ m\in T_1,j\in T_2\\  \lambda\in \Lambda}}|a_{\lambda,m,j}|^2.
\end{align}	
Using \eqref{c},  we compute
\begin{align}\label{bb}
&||T\{a_{\lambda,m,j}\}||_{\ell^2}^2 \nonumber\\
&=  \Big|\Big|\sum\limits_{l\in\mathbb{Z}}c_l E_{\frac{m_0}{M}}R_{4Nl}\bold{W}_{j_0}\Big|\Big|_{\ell^2}^2\nonumber\\&=\Big|\Big|\mathcal{F}\Big(\sum\limits_{l\in\mathbb{Z}}c_l E_{\frac{m_0}{M}}R_{4Nl}\bold{W}_{j_0}\Big)\Big|\Big|_{L^2}^2\nonumber\\
&=\Big|\Big|\sum\limits_{l\in\mathbb{Z}}c_l e^{4\pi i N(2l)(\frac{m_0}{M}+\xi)}\mathcal{F}(E_{\frac{m_0}{M}}\bold{W}_{j_0})(\xi) \Big|\Big|_{L^2}^2\nonumber\\
&=\int\limits_{\Omega}\sum\limits_{k=1}^S\Big|\sum\limits_{l\in\mathbb{Z}}c_{l}e^{4\pi i N(2l)(\frac{m_0}{M}+\xi)}\Big[\mathcal{F}(E_{\frac{m_0}{M}}\bold{W}_{j_0})(\xi)\Big]_k\Big|^2\,d\xi\nonumber\\
&=\sum\limits_{k=1}^S\int\limits_{0}^{\frac{1}{2}}\Big|\sum\limits_{l\in \mathbb{Z}}c_{l}e^{4\pi i N(2l)(\frac{m_0}{M}+\xi)}\Big|^2\Bigg(\Big|\Big[\mathcal{F}(E_{\frac{m_0}{M}}\bold{W}_{j_0})(\xi)\Big]_k\Big|^2
+\Big|\Big[\mathcal{F}(E_{\frac{m_0}{M}}\bold{W}_{j_0})\big(\xi+\frac{N}{2}\big)\Big]_k\Big|^2\Bigg)\,d\xi\nonumber\\
&=\sum\limits_{k=1}^S\sum\limits_{t=0}^{2N-1}\int\limits_{0}^{\frac{1}{4N}}\Big|\sum\limits_{l\in \mathbb{Z}}c_{l}e^{4\pi i N(2l)(\frac{m_0}{M}+\xi)}\Big|^2\Bigg(\Big|\Big[\mathcal{F}(E_{\frac{m_0}{M}}\bold{W}_{j_0})(\xi+\frac{t}{4N})\Big]_k\Big|^2\nonumber\\
& \quad +\Big|\Big[\mathcal{F}(E_{\frac{m_0}{M}}\bold{W}_{j_0})\Big(\xi+\frac{N}{2}+\frac{t}{4N}\Big)\Big]_k\Big|^2\Bigg)\,d\xi.
\end{align}
		Using \eqref{aa123}, \eqref{c}, \eqref{d} and \eqref{bb}, we have
\begin{align*}
&\sum\limits_{k=1}^S\sum\limits_{t=0}^{2N-1}\int\limits_{0}^{\frac{1}{4N}}|f(\xi)|^2\Bigg(\Big|\Big[\mathcal{F}(E_{\frac{m_0}{M}}\bold{W}_{j_0})(\xi+\frac{t}{4N})\Big]_k\Big|^2
+\Big|\Big[\mathcal{F}(E_{\frac{m_0}{M}}\bold{W}_{j_0})\Big(\xi+\frac{N}{2}+\frac{t}{4N}\Big)\Big]_k\Big|^2\Bigg)\,d\xi\\
&\le 4NB_0 \int\limits_{0}^{\frac{1}{4N}}|f(\xi)|^2\,d\xi.
\end{align*}
	As $f(\xi)$ is an arbitrary element of $L^2(0,\frac{1}{4N})$, we conclude that for $a.e.\, \xi\in [0,\frac{1}{4N}[$,
	\begin{align*}
	\sum\limits_{k=1}^S\sum\limits_{t=0}^{2N-1}\Big|\Big[\mathcal{F}(E_{\frac{m_0}{M}}\bold{W}_{j_0})(\xi+\frac{t}{4N})\Big]_k\Big|^2
	&+\Big|\Big[\mathcal{F}(E_{\frac{m_0}{M}}\bold{W}_{j_0})\Big(\xi+\frac{N}{2}+\frac{t}{4N}\Big)\Big]_k\Big|^2 \le 4NB_0,
	\end{align*}
which entails
	\begin{align*}
	\sum\limits_{k=1}^S \Big|\Big[\mathcal{F}(E_{\frac{m}{M}}\bold{W}_{j})(\xi)\Big]_k\Big|^2 \le 4NB_0 \,\,\,\,\text{for}\,\,a.e.\,\xi\in \Omega,m\in T_1,j\in T_2.
	\end{align*}
This concludes the proof.	
\end{proof}
By Theorem \ref{th1} and Theorem \ref{th2}, we obtain the following characterization of DVNUG Bessel sequences.
\begin{cor}\label{456}
Let $\{\bold{W}_j\}_{j\in T_2} \subset \ell^2(\Lambda,\mathbb{C}^S)$. Then,  $\mathcal{G}(\Lambda,  T_1, T_2, \bold{W}_j)$  is a  DVNUG Bessel sequence if and only if there exists a constant $C_0$ such that for a.e. $\xi\in \Omega,  m\in T_1$, $j\in T_2$, we have
	\begin{align}\label{Bessel}
	\Big|\Big|\mathcal{F}(E_{\frac{m}{M}}\bold{W}_j)(\xi)\Big|\Big|_{\mathbb{C}^S}\le C_0<\infty.
	\end{align}
\end{cor}
\begin{rem}
Condition \eqref{Bessel} does not guarantee that $\mathcal{G}(\Lambda,  T_1, T_2, \bold{W}_j)$  constitutes a  DVNUG frame for $\ell^2(\Lambda,\mathbb{C}^S)$. Indeed, if we choose $M=1,P=1,S=2$ and $N$, $r$ as  standard, then $T_1=\{0\},T_2=\{0,1\}$,  $\Lambda=\{0,\frac{r}{N}\}+2\mathbb{Z}$  and $\Omega=[0,\frac{1}{2}[\cup[\frac{N}{2},\frac{N+1}{2}[$. Define $\{\bold{W}_j \}_{j\in T_2}\in \ell^2(\Lambda,\mathbb{C}^2)$ in term of its Fourier transforms as follows:
	\begin{align*}
	\mathcal{F}(\bold{W}_0)=\begin{bmatrix}
	\sqrt{2N}\chi_{[0,\frac{1}{4N}[}(\xi)
	\vspace{5pt}\\
	0
	\end{bmatrix} \  \quad \text{and} \ \quad
	\mathcal{F}(\bold{W}_1)=\begin{bmatrix}
	0
	\vspace{5pt}\\
	\sqrt{2N}\chi_{[0,\frac{1}{4N}[ }(\xi)
	\end{bmatrix}.	
	\end{align*}
	As $T_1=\{0\}$, so it can be  easily checked, for a.e. $\xi\in \Omega$,  $m\in T_1$, $j\in T_2$ that
	\begin{align*}
	\Big|\Big|\mathcal{F}(E_{\frac{m}{M}}\bold{W}_j)(\xi)\Big|\Big|_{\mathbb{C}^2}\le \sqrt{2N} <\infty.
	\end{align*}
	By Corollary \ref{456}, $\mathcal{G}(\Lambda,  T_1, T_2, \bold{W}_j)=\{E_{\frac{m}{M}}R_{2N\lambda}\bold{W}_j:m\in T_1,j\in T_2,\lambda\in \Lambda \}$
	is DVNUG Bessel sequence for $\ell^2(\Lambda,\mathbb{C}^2)$. If possible,  assume that  $\mathcal{G}(\Lambda,  T_1, T_2, \bold{W}_j)$  is a   DVNUG frame for $\ell^2(\Lambda,\mathbb{C}^2)$, then there exists a positive constant $A_0$ (say),  such that for any $\bold{Z}\in \ell^2(\Lambda,\mathbb{C}^2)$, we have
	\begin{align}\label{3.9}
	A_0 ||\mathcal{F}(\bold{Z})||^2=A_0 ||\bold{Z}||^2 &\le \sum\limits_{\substack{ m\in T_1,j\in T_2\\  \lambda\in \Lambda}}\Big|\langle \bold{Z},E_{\frac{m}{M}}R_{2N\lambda}\bold{W}_j\rangle\Big|^2.
	\end{align}
We compute
\begin{align}\label{3.10}
&\sum\limits_{\substack{ m\in T_1,j\in T_2\\  \lambda\in \Lambda}}\Big|\langle \bold{Z},E_{\frac{m}{M}}R_{2N\lambda}\bold{W}_j\rangle\Big|^2 \nonumber\\
&=\sum\limits_{\substack{ j\in T_2,  \lambda\in \Lambda}}\Big|\langle \mathcal{F}(\bold{Z}),\mathcal{F}(R_{2N\lambda}\bold{W}_j)\rangle\Big|^2\nonumber\\
&=\sum\limits_{\substack{ j\in T_2,  \lambda\in \Lambda}}\Big|\langle \mathcal{F}(\bold{Z})(\xi), e^{4\pi i N\lambda\xi} \mathcal{F}(\bold{W}_j)(\xi)\rangle\Big|^2\nonumber\\
&=\sum\limits_{\lambda\in \Lambda}\Big|\int_{\Omega}[\mathcal{F}(\bold{Z})(\xi)]_1e^{-4\pi i N\lambda\xi}\sqrt{2N}\chi_{[0,\frac{1}{4N}[}(\xi)\,d\xi \Big|^2
+\sum\limits_{\lambda\in \Lambda}\Big|\int_{\Omega}[\mathcal{F}(\bold{Z})(\xi)]_2e^{-4\pi i N\lambda\xi}\sqrt{2N}\chi_{[0,\frac{1}{4N}[}(\xi)\,d\xi \Big|^2\nonumber\\
&=2N\Bigg(\sum\limits_{l\in \mathbb{Z}}\Big|\int_{0}^{\frac{1}{4N}}[\mathcal{F}(\bold{Z})(\xi)]_1e^{-2\pi i (4N)l\xi}\,d\xi \Big|^2
+\sum\limits_{l\in \mathbb{Z}}\Big|\int_{0}^{\frac{1}{4N}}[\mathcal{F}(\bold{Z})(\xi)]_1e^{-4\pi ir \xi} e^{-2\pi i (4N)l\xi}\,d\xi \Big|^2\nonumber\\
& \quad +\sum\limits_{l\in \mathbb{Z}}\Big|\int_{0}^{\frac{1}{4N}}[\mathcal{F}(\bold{Z})(\xi)]_2e^{-2\pi i (4N)l\xi}\,d\xi \Big|^2
+\sum\limits_{l\in \mathbb{Z}}\Big|\int_{0}^{\frac{1}{4N}}[\mathcal{F}(\bold{Z})(\xi)]_2e^{-4\pi ir \xi} e^{-2\pi i (4N)l\xi}\,d\xi \Big|^2
	\Bigg).
	\end{align}
Using Lemma \ref{naya} and  \eqref{3.10}, we have
\begin{align}\label{3.11}
&\sum\limits_{\substack{ m\in T_1,j\in T_2\\  \lambda\in \Lambda}}\Big|\langle \bold{Z},E_{\frac{m}{M}}R_{2N\lambda}\bold{W}_j\rangle\Big|^2 \nonumber\\
&= \frac{1}{2}\Bigg(\int\limits_{0}^{\frac{1}{4N}}\big|[\mathcal{F}(\bold{Z})(\xi)]_1\big|^2d\xi+\int\limits_{0}^{\frac{1}{4N}}\big|[\mathcal{F}(\bold{Z})(\xi)]_1\big|^2d\xi
+\int\limits_{0}^{\frac{1}{4N}}\big|[\mathcal{F}(\bold{Z})(\xi)]_2\big|^2d\xi+\int\limits_{0}^{\frac{1}{4N}}\big|[\mathcal{F}(\bold{Z})(\xi)]_2\big|^2d\xi\Bigg) \nonumber\\
&=\int\limits_{0}^{\frac{1}{4N}}\big|[\mathcal{F}(\bold{Z})(\xi)]_1\big|^2d\xi+\int\limits_{0}^{\frac{1}{4N}}\big|[\mathcal{F}(\bold{Z})(\xi)]_2\big|^2d\xi.
\end{align}
 Now, if we choose $\bold{Z}_0\in \ell^2(\Lambda,\mathbb{C}^2)$ in term of of its Fourier transformation as follows:
\begin{align}\label{godimedia}
\mathcal{F}(\bold{Z}_0)(\xi)=\begin{bmatrix}
\chi_{[0,\frac{1}{4N}[}+\frac{1}{A_0}\chi_{[\frac{1}{4N},\frac{2}{4N}[}
\vspace{5pt}\\
\chi_{[0,\frac{1}{4N}[}+\frac{1}{A_0}\chi_{[\frac{1}{4N},\frac{2}{4N}[}
\end{bmatrix},
\end{align}
then by \eqref{godimedia}, \eqref{3.9} and \eqref{3.11},  we arrive at a  contradiction.
\end{rem}
The following lemma plays a crucial role in  characterization of DVNUG frames for discrete vector-valued nonuniform signal spaces.

\begin{lem}\label{Lemma1}
	Assume $\{\bold{W}_j \}_{j\in T_2}\subset \ell^2(\Lambda,\mathbb{C}^S)$. For $m\in T_1$, $j\in T_2$, $k\in \{1,2,\dots,S\}$ and $\xi\in [0,\frac{1}{4N}[$,  define matrices
\begin{align*}
{\mathcal{A}}_{j,m,k}(\xi), \quad {\mathcal{B}}_{j,m,k}(\xi), \quad \text{and} \quad \ {\mathcal{M}}_{m,k}(\xi),
\end{align*}
respectively, as follows.
		\begin{align*}
	{\mathcal{A}}_{j,m,k}(\xi)=\tiny{\begin{bmatrix}
	\Big[\mathcal{F}(E_{\frac{m}{M}}\bold{W}_j)(\xi)\Big]_k
	\vspace{5pt}\\
	\Big[\mathcal{F}(E_{\frac{m}{M}}\bold{W}_j)(\xi+\frac{1}{4N})\Big]_k\\
	\vdots
	\vspace{5pt}\\
\Big[\mathcal{F}(E_{\frac{m}{M}}\bold{W}_j)(\xi+\frac{2N-1}{4N})\Big]_k
\vspace{5pt}\\
		\Big[\mathcal{F}(E_{\frac{m}{M}}\bold{W}_j)(\xi+\frac{N}{2})\Big]_k
		\vspace{5pt}\\
				\Big[\mathcal{F}(E_{\frac{m}{M}}\bold{W}_j)(\xi+\frac{N}{2}+\frac{1}{4N})\Big]_k\\
		\vdots
		\vspace{5pt}\\
		\Big[\mathcal{F}(E_{\frac{m}{M}}\bold{W}_j)(\xi+\frac{N}{2}+\frac{2N-1}{4N})\Big]_k
	\end{bmatrix}_{4N\times 1}},\,\,
\mathcal{B}_{j,m,k}(\xi)=\tiny{\begin{bmatrix}
	e^{4\pi ir\xi}\Big[\mathcal{F}(E_{\frac{m}{M}}\bold{W}_j)(\xi)\Big]_k
	\vspace{5pt}\\
	e^{4\pi ir(\xi+\frac{1}{4N})}\Big[\mathcal{F}(E_{\frac{m}{M}}\bold{W}_j)(\xi+\frac{1}{4N})\Big]_k\\
	\vdots
	\vspace{5pt}\\	
	e^{4\pi ir(\xi+\frac{2N-1}{4N})}\Big[\mathcal{F}(E_{\frac{m}{M}}\bold{W}_j)(\xi+\frac{2N-1}{4N})\Big]_k
	\vspace{5pt}\\
	e^{4\pi ir\xi}\Big[\mathcal{F}(E_{\frac{m}{M}}\bold{W}_j)(\xi+\frac{N}{2})\Big]_k
	\vspace{5pt}\\
	e^{4\pi ir(\xi+\frac{1}{4N})}\Big[\mathcal{F}(E_{\frac{m}{M}}\bold{W}_j)(\xi+\frac{N}{2}+\frac{1}{4N})\Big]_k\\
	\vdots
	\vspace{5pt}\\
	e^{4\pi ir(\xi+\frac{2N-1}{4N})}\Big[\mathcal{F}(E_{\frac{m}{M}}\bold{W}_j)(\xi+\frac{N}{2}+\frac{2N-1}{4N})\Big]_k
	\end{bmatrix}_{4N\times 1}}
	\end{align*}
	and,
	\begin{align}\label{M}
	\mathcal{M}_{m,k}(\xi)=	\begin{bmatrix}
	\mathcal{A}_{0,m,k}(\xi) &\mathcal{B}_{0,m,k}(\xi) &\mathcal{A}_{1,m,k}(\xi) &\mathcal{B}_{1,m,k}(\xi)\cdots &\mathcal{A}_{P,m,k}(\xi) &\mathcal{B}_{P,m,k}(\xi)
	\end{bmatrix}_{4N\times2(P+1)}.
	\end{align}
For $m\in T_1$, $j\in T_2$ and a.e. $\xi\in \Omega$, assume that
\begin{align}\label{33}
	  \Big|\Big|\mathcal{F}(E_{\frac{m}{M}}\bold{W}_j)(\xi)\Big|\Big|_{\mathbb{C}^S}\le B_0<\infty,
\end{align}
for some constant $B_0$. Then,  for any $\bold{Z}\in \ell^2(\Lambda,\mathbb{C}^S)$, we have
	  \begin{align*}
	  4N\sum\limits_{\substack{ m\in T_1,j\in T_2\\  \lambda\in \Lambda}}\Big|\langle \bold{Z},E_{\frac{m}{M}}R_{2N\lambda}\bold{W}_j\rangle\Big|^2=\sum\limits_{m\in T_1}\int\limits_{0}^{\frac{1}{4N}}
	  \Big|\Big|\sum\limits_{k=1}^{S}\mathcal{M}_{m,k}^*(\xi)\mathcal{V}_{\bold{Z}_k}(\xi)   \Big|\Big|_{\mathbb{C}^{2(P+1)}}^2\,d\xi,
	  \end{align*}
	 where  $\mathcal{M}_{m,k}^*(\xi)$ is conjugate transpose of matrix $\mathcal{M}_{m,k}(\xi)$ and
	 \begin{align}\label{hari}
	 \mathcal{V}_{\bold{Z}_k}(\xi)=\begin{bmatrix}
	 \big[\mathcal{F}(\bold{Z})(\xi)   \big]_k
	 \vspace{5pt}\\
	 \big[\mathcal{F}(\bold{Z})(\xi+\frac{1}{4N})   \big]_k\\
	 \vdots
	  \vspace{5pt}\\
	 \big[\mathcal{F}(\bold{Z})(\xi+\frac{2N-1}{4N})   \big]_k
	 \vspace{5pt}\\
	 \big[\mathcal{F}(\bold{Z})(\xi+\frac{N}{2})   \big]_k
	 \vspace{5pt}\\
	 \big[\mathcal{F}(\bold{Z})(\xi+\frac{N}{2}+\frac{1}{4N})   \big]_k\\
	 \vdots
	 \vspace{5pt}\\
	 \big[\mathcal{F}(\bold{Z})(\xi+\frac{N}{2}+\frac{2N-1}{4N})   \big]_k\\
	 \end{bmatrix}_{4N\times 1}   ,\,\,\,\text{for}\,\, k\in \{1,2,\dots,S\}.
	 \end{align}
	
\end{lem}
\begin{proof}
	Let $\bold{Z}\in \ell^2(\Lambda,\mathbb{C}^S)$ be  arbitrary. For $m\in T_1$, $j\in T_2$ and $\xi \in [0,\frac{1}{2}[$, define
	\begin{align}\label{X}
	X_{m,j}(\xi)=\sum\limits_{k=1}^S\Bigg (\big[\mathcal{F}(\bold{Z}(\xi)) \big]_k\overline{\big[\mathcal{F}(E_{\frac{m}{M}}\bold{W}_j)(\xi) \big]_k}+\big[\mathcal{F}(\bold{Z}(\xi+\frac{N}{2})) \big]_k\overline{\big[\mathcal{F}(E_{\frac{m}{M}}\bold{W}_j)(\xi+\frac{N}{2}) \big]_k}\Bigg).
	\end{align}
Then,  for $\xi \in [0,\frac{1}{4N}[$, $m\in T_1$,  it is easy to see  that
	 \begin{align*}
	 \sum\limits_{k=1}^S \mathcal{M}_{m,k}^*(\xi)\mathcal{V}_{\bold{Z}_k}(\xi)=\begin{bmatrix}
	 \sum\limits_{t=0}^{2N-1}X_{m,0}(\xi+\frac{t}{4N})
	 \vspace{5pt}\\
	 \sum\limits_{t=0}^{2N-1} e^{-4\pi i r(\xi+\frac{t}{4N})} X_{m,0}(\xi+\frac{t}{4N})
	 \vspace{5pt}\\
	 \sum\limits_{t=0}^{2N-1}X_{m,1}(\xi+\frac{t}{4N})
	 \vspace{5pt}\\
	 \sum\limits_{t=0}^{2N-1} e^{-4\pi i r(\xi+\frac{t}{4N})} X_{m,1}(\xi+\frac{t}{4N})\\
	 \vdots
	 \vspace{5pt}\\
	 \sum\limits_{t=0}^{2N-1}X_{m,P}(\xi+\frac{t}{4N})
	 \vspace{5pt}\\
	 \sum\limits_{t=0}^{2N-1} e^{-4\pi i r(\xi+\frac{t}{4N})} X_{m,P}(\xi+\frac{t}{4N})
	 \end{bmatrix}_{2(P+1)\times 1}.
	 	 \end{align*}
	 Also,
 \begin{align}\label{22}
&\sum\limits_{m\in T_1}\int\limits_{0}^{\frac{1}{4N}}\Big|\Big|\sum\limits_{k=1}^S \mathcal{M}_{m,k}^*(\xi)\mathcal{V}_{\bold{Z}_k}(\xi)\Big|\Big|_{\mathbb{C}^{2(P+1)}}^2\,d\xi\\
&= \sum\limits_{m\in T_1}\sum\limits_{j\in T_2}\int\limits_{0}^{\frac{1}{4N}}\Bigg(\Big|\sum\limits_{t=0}^{2N-1}X_{m,j}(\xi+\frac{t}{4N})\Big|^2\nonumber
+\Big|\sum\limits_{t=0}^{2N-1}e^{-\pi i r\frac{t}{N}} X_{m,j}(\xi+\frac{t}{4N})\Big|^2\Bigg)\,d\xi.
\end{align}	
	Now consider,
\begin{align}\label{Y}
&\sum\limits_{\substack{ m\in T_1,j\in T_2\\  \lambda\in \Lambda}}\Big|\langle \bold{Z},E_{\frac{m}{M}}R_{2N\lambda}\bold{W}_j\rangle\Big|^2 \nonumber\\
&=\sum\limits_{\substack{ m\in T_1,j\in T_2\\  \lambda\in \Lambda}}\Big|\langle \mathcal{F}(\bold{Z})(\xi),e^{4\pi i N\lambda(\frac{m}{M}+\xi)} \mathcal{F}(E_{\frac{m}{M}}\bold{W}_j)(\xi)\rangle\Big|^2\nonumber\\
	&=\sum\limits_{\substack{ m\in T_1,j\in T_2\\  \lambda\in \Lambda}}\Big|\int\limits_{\Omega}\sum\limits_{k=1}^S[\mathcal{F}(\bold{Z})(\xi)]_k\overline{[ \mathcal{F}(E_{\frac{m}{M}}\bold{W}_j)(\xi) ]_k}e^{-4\pi i N\lambda(\frac{m}{M}+\xi)}d\xi\Big|^2\nonumber\\
	&=\sum\limits_{\substack{ m\in T_1,j\in T_2\\  \l\in \mathbb{Z}}}\Big|\int\limits_{\Omega}\sum\limits_{k=1}^S[\mathcal{F}(\bold{Z})(\xi)]_k\overline{[ \mathcal{F}(E_{\frac{m}{M}}\bold{W}_j)(\xi) ]_k}e^{-4\pi i N(2l)(\frac{m}{M}+\xi)}d\xi\Big|^2\nonumber\\
	&+\sum\limits_{\substack{ m\in T_1,j\in T_2\\  \l\in \mathbb{Z}}}\Big|\int\limits_{\Omega}\sum\limits_{k=1}^S[\mathcal{F}(\bold{Z})(\xi)]_k\overline{[ \mathcal{F}(E_{\frac{m}{M}}\bold{W}_j)(\xi) ]_k}e^{-4\pi i N(\frac{r}{N}+2l)(\frac{m}{M}+\xi)}d\xi\Big|^2.
		\end{align}
Using  \eqref{X} and \eqref{Y}, we have	
\begin{align}\label{aa}
&\sum\limits_{\substack{ m\in T_1,j\in T_2\\  \lambda\in \Lambda}}\Big|\langle \bold{Z},E_{\frac{m}{M}}R_{2N\lambda}\bold{W}_j\rangle\Big|^2 \nonumber\\
&=\sum\limits_{\substack{ m\in T_1,j\in T_2\\  \l\in \mathbb{Z}}}\Big|\int\limits_{0}^{\frac{1}{2}}X_{m,j}(\xi)e^{-4\pi i N(2l)(\frac{m}{M}+\xi)}d\xi\Big|^2\nonumber
+\sum\limits_{\substack{ m\in T_1,j\in T_2\\  \l\in \mathbb{Z}}}\Big|\int\limits_{0}^{\frac{1}{2}}X_{m,j}(\xi)e^{-4\pi i N(\frac{r}{N}+2l)(\frac{m}{M}+\xi)}\,d\xi\Big|^2\nonumber\\
&=\sum\limits_{\substack{ m\in T_1,j\in T_2\\  \l\in \mathbb{Z}}}\Big|\int\limits_{0}^{\frac{1}{4N}}\sum\limits_{t=0}^{2N-1}X_{m,j}(\xi+\frac{t}{4N})e^{-2\pi i (4N)l(\frac{m}{M}+\xi)}d\xi\Big|^2\nonumber\\
	&+\sum\limits_{\substack{ m\in T_1,j\in T_2\\  \l\in \mathbb{Z}}}\Big|\int\limits_{0}^{\frac{1}{4N}}\sum\limits_{t=0}^{2N-1}X_{m,j}(\xi+\frac{t}{4N}) e^{-4\pi ir(\frac{m}{M}+\xi+\frac{t}{4N})}  e^{-2\pi i (4N)l(\frac{m}{M}+\xi)}d\xi\Big|^2.
		\end{align}
By invoking  Lemma \ref{naya},  Ineq. \eqref{33}  and  \eqref{aa}, we have
\begin{align}\label{11}
&\sum\limits_{\substack{ m\in T_1,j\in T_2\\  \lambda\in \Lambda}}\Big|\langle \bold{Z},E_{\frac{m}{M}}R_{2N\lambda}\bold{W}_j\rangle\Big|^2\nonumber\\
&=\frac{1}{4N}\Bigg(\sum\limits_{\substack{ m\in T_1 \atop j\in T_2}}\int\limits_{0}^{\frac{1}{4N}}\Big|\sum\limits_{t=0}^{2N-1}X_{m,j}(\xi+\frac{t}{4N})\Big|^2\,d\xi
+\sum\limits_{\substack{ m\in T_1 \atop j\in T_2}}\int\limits_{0}^{\frac{1}{4N}}\Big|\sum\limits_{t=0}^{2N-1}X_{m,j}(\xi+\frac{t}{4N})e^{-i \pi \frac{r}{N}t}\Big|^2\,d\xi\Bigg).
	\end{align}
 Equations  \eqref{22} and \eqref{11} give  the desired result.
	\end{proof}

The following result characterizes  DVNUG frames in discrete vector-valued nonuniform  signal spaces.
\begin{thm}\label{main1}
Suppose  $\{\bold{W}_j\}_{j\in T_2} \subset \ell^2(\Lambda,\mathbb{C}^S)$ satisfies
\begin{align*}
\Big|\Big|\mathcal{F}(E_{\frac{m}{M}}\bold{W}_j)(\xi)\Big|\Big|_{\mathbb{C}^S}\le B_0<\infty \  \ \text{for} \ m \in T_1,  \   j \in T_2,  \ \text{a.e.} \ \xi\in \Omega,
\end{align*}
where  $B_0$ is a positive constant. Let $\mathcal{M}_{m,k}(\xi)$ be $4\times 2(P+1)$ matrices for  $m\in T_1,j\in T_2, a.e.\,\xi \in [0,\frac{1}{4N}[ $, as defined in \eqref{M}. Then,
 $\mathcal{G}(\Lambda,  T_1, T_2, \bold{W}_j)$  is a DVNUG frame for discrete vector-valued nonuniform signal  space $\ell^2(\Lambda,\mathbb{C}^S)$ with lower frame bound $A_0>0 $  if and only if
	\begin{align}\label{mama}
	\sum\limits_{m\in T_1}\Big|\Big|\sum\limits_{k=1}^S\mathcal{M}_{m,k}^*(\xi)\bold{C}_k\Big|\Big|_{\mathbb{C}^{2(P+1)}}^2\ge 4N A_0 \sum\limits_{k=1}^S ||\bold{C}_k||_{\mathbb{C}^{4N}}^2,\,\,a.e.\,\,\xi\in \Big(0,\frac{1}{4N}\Big)
	\end{align}
	for all $\bold{C}_k\in \mathbb{C}^{4N}$, $ k=1,2,\dots,S$.
	\end{thm}

\begin{proof}
	Since $\Big|\Big|\mathcal{F}(E_{\frac{m}{M}}\bold{W}_j)(\xi)\Big|\Big|_{\mathbb{C}^S}\le B_0<\infty$, for $m \in T_1$, $j \in T_2$,  a.e. $\xi \in \Omega$,   so by Lemma \ref{Lemma1}, for each $\bold{Z}\in\ell^2(\Lambda,\mathbb{C}^S)$, we have,
	\begin{align}\label{no}
	4N\sum\limits_{\substack{ m\in T_1,j\in T_2\\  \lambda\in \Lambda}}\Big|\langle \bold{Z},E_{\frac{m}{M}}R_{2N\lambda}\bold{W}_j\rangle\Big|^2=\sum\limits_{m\in T_1}\int\limits_{0}^{\frac{1}{4N}}
	\Big|\Big|\sum\limits_{k=1}^{S}\mathcal{M}_{m,k}^*(\xi)\mathcal{V}_{\bold{Z}_k}(\xi)   \Big|\Big|_{\mathbb{C}^{2(P+1)}}^2\,d\xi,
	\end{align}	
	where $\mathcal{V}_{\bold{Z}_k}$ defined as in  \eqref{hari}.

	Assume first that  $\mathcal{G}(\Lambda,  T_1, T_2, \bold{W}_j)$  is  a DVNUG frame for $\ell^2(\Lambda,\mathbb{C}^S)$ with lower frame bound $A_0$, say. Let $f(\xi)\in L^2(0,\frac{1}{4N})$ be an arbitrary element. For $k\in\{1,2,\dots,S\}$, let
\begin{align}\label{C}
\bold{C}_k=\begin{bmatrix}
[\bold{C}_k]_1
\vspace{5pt}\\
[\bold{C}_k]_2\\
\vdots
\vspace{5pt}\\
[\bold{C}_k]_{2N}
\vspace{5pt}\\
\,\,\,\,\,\,[\bold{C}_k]_{2N+1}\\
\vdots
\vspace{5pt}\\
[\bold{C}_k]_{4N}
\end{bmatrix}_{4N\times 1} \in \mathbb{C}^{4N}.
\end{align}
 Define  $\bold{Z}\in \ell^2(\Lambda,\mathbb{C}^S)$ in term of its Fourier transform, $\mathcal{F}(\bold{Z}(\xi) )=\begin{bmatrix}
 [\mathcal{F}(\bold{Z}(\xi) )]_1
 \vspace{5pt}\\
 [\mathcal{F}(\bold{Z}(\xi) )]_2\\
 \vdots
 \vspace{5pt}\\
 [\mathcal{F}(\bold{Z}(\xi) )]_S
 \end{bmatrix}\in L^2(\Omega,\mathbb{C}^S)$, as follows:\\
 For any fixed  $k_0\in \{1,2,\dots,S\}$,
 if $\xi \in [0,\frac{1}{2}[$, then there exists unique $t'\in \{0,1,\dots,2N-1\}$ such that $\xi-\frac{t'}{4N} \in [0,\frac{1}{4N}[$, so define
\begin{align*}
[\mathcal{F}(\bold{Z})(\xi)]_{k_0}=[\bold{C}_{k_0}]_{t'+1}f(\xi-\frac{t'}{4N});
\end{align*}
and if $\xi\in [\frac{N}{2},\frac{N+1}{2}[$, there exists $t''\in\{0,1,\dots,2N-1\}$ such that $\xi-\frac{t''}{4N}-\frac{N}{2} \in [0,\frac{1}{4N}[$, so define
\begin{align*}
[\mathcal{F}(\bold{Z})(\xi))]_{k_0}=[\bold{C}_{k_0}]_{2N+1+t''}f(\xi-\frac{t''}{4N}-\frac{N}{2}).
\end{align*}
   Thus,  for $k\in \{1,2,\dots,S\}, t\in \{0,1,\dots,2N-1 \}$ and $\xi\in [0,\frac{1}{4N}[$, we have
   \begin{align}
   \Big[\mathcal{F}(\bold{Z})(\xi+\frac{t}{4N})\Big]_k&=[\bold{C}_k]_{t+1}f(\xi),\label{C1}\\
   \Big[\mathcal{F}(\bold{Z})(\xi+\frac{t}{4N}+\frac{N}{2})\Big]_k&=[\bold{C}_k]_{2N+t+1}f(\xi).\label{C2}
      \end{align}
Now,  for $k\in \{1,2,\dots,S\}$ and $\xi\in [0,\frac{1}{4N}[$, we define $\mathcal{V}_{\bold{Z}_k}(\xi)$ as follows
\begin{align}\label{C3}
\mathcal{V}_{\bold{Z}_k}(\xi)=\begin{bmatrix}
\big[\mathcal{F}(\bold{Z})(\xi)   \big]_k
\vspace{5pt}\\
\big[\mathcal{F}(\bold{Z})(\xi+\frac{1}{4N})   \big]_k\\
\vdots
\vspace{5pt}\\
\big[\mathcal{F}(\bold{Z})(\xi+\frac{2N-1}{4N})   \big]_k
\vspace{5pt}\\
\big[\mathcal{F}(\bold{Z})(\xi+\frac{N}{2})   \big]_k
\vspace{5pt}\\
\big[\mathcal{F}(\bold{Z})(\xi+\frac{N}{2}+\frac{1}{4N})   \big]_k\\
\vdots
\vspace{5pt}\\
\big[\mathcal{F}(\bold{Z})(\xi+\frac{N}{2}+\frac{2N-1}{4N})   \big]_k\\
\end{bmatrix}_{4N\times 1}.
\end{align}
By \eqref{C}, \eqref{C1},  \eqref{C2} and \eqref{C3}, for $k\in \{1,2,\dots,S\}$ and $\xi \in [0,\frac{1}{4N}[$,  we have
\begin{align}\label{V}
\mathcal{V}_{\bold{Z}_k}(\xi)=\bold{C}_kf(\xi).
\end{align}
Now, using \eqref{V} and \eqref{no},  we have
\begin{align}\label{yes}
\int\limits_{0}^{\frac{1}{4N}}|f(\xi)|^2\sum\limits_{m\in T_1}\Big|\Big|\sum\limits_{k=1}^S\mathcal{M}_{m,k}^*(\xi)\bold{C}_k\Big|\Big|_{\mathbb{C}^{2(P+1)}}^2\,d\xi&=\int\limits_{0}^{\frac{1}{4N}}\sum\limits_{m\in T_1}\Big|\Big|\sum\limits_{k=1}^S\mathcal{M}_{m,k}^*(\xi)\mathcal{V}_{\bold{Z}_k}(\xi)\Big|\Big|_{\mathbb{C}^{2(P+1)}}^2\,d\xi\nonumber\\
&=4N\sum\limits_{\substack{ m\in T_1,j\in T_2\\  \lambda\in \Lambda}}\Big|\langle \bold{Z},E_{\frac{m}{M}}R_{2N\lambda}\bold{W}_j\rangle\Big|^2.
\end{align}
Therefore,  by \eqref{yes} and the hypothesis that  $\mathcal{G}(\Lambda,  T_1, T_2, \bold{W}_j)$  is a frame with lower frame bound $A_0$, we get
\begin{align}\label{papa}
&\int\limits_{0}^{\frac{1}{4N}}|f(\xi)|^2\sum\limits_{m\in T_1}\Big|\Big|\sum\limits_{k=1}^S\mathcal{M}_{m,k}^*(\xi)\bold{C}_k\Big|\Big|_{\mathbb{C}^{2(P+1)}}^2\,d\xi \nonumber\\
&\ge 4NA_0||\bold{Z}||_{\ell^2}^2\nonumber\\
&=4NA_0||\mathcal{F}(\bold{Z})||_{L^2}^2\nonumber\\
&=4NA_0\sum\limits_{k=1}^S\sum\limits_{t=0}^{2N-1}\Bigg(\int\limits_{0}^{4N}\Big(\Big|\Big[\mathcal{F}(\bold{Z})(\xi+\frac{t}{4N}) \Big]_k\Big|^2+\Big|\Big[\mathcal{F}(\bold{Z})(\xi+\frac{t}{4N}+\frac{N}{2}) \Big]_k\Big|^2\Bigg)\,d\xi.
\end{align}
Using \eqref{C1},  \eqref{C2} and \eqref{papa}, we have
\begin{align*}
\int\limits_{0}^{\frac{1}{4N}}|f(\xi)|^2\sum\limits_{m\in T_1}\Big|\Big|\sum\limits_{k=1}^S\mathcal{M}_{m,k}^*(\xi)\bold{C}_k\Big|\Big|_{\mathbb{C}^{2(P+1)}}^2\,d\xi \ge 4NA_0 \sum\limits_{k=1}^S\int\limits_{0}^{\frac{1}{4N}} ||\bold{C}_k||_{\mathbb{C}^{4N}}^2 |f(\xi)|^2\,d\xi.
\end{align*}
Since $f(\xi)$ is an arbitrary element of $L^2(0,\frac{1}{4N})$, we conclude for $a.e.\,\xi\in (0,\frac{1}{4N})$,  that
\begin{align*}
\sum\limits_{m\in T_1}\Big|\Big|\sum\limits_{k=1}^S\mathcal{M}_{m,k}^*(\xi)\bold{C}_k\Big|\Big|_{\mathbb{C}^{2(P+1)}}^2 \ge 4NA_0\sum\limits_{k=1}^S  ||\bold{C}_k||_{\mathbb{C}^{4N}}^2.
\end{align*}

Conversely, assume that  \eqref{mama} holds. As $\Big|\Big|\mathcal{F}(E_{\frac{m}{M}}\bold{W}_j)(\xi)\Big|\Big|_{\mathbb{C}^S}\le B_0<\infty$, for $m\in T_1$, $j\in T_2$,  a.e. $\xi\in \Omega$, so by Theorem \ref{th1}, $\mathcal{G}(\Lambda,  T_1, T_2, \bold{W}_j)$  form a  DVNUG Bessel sequence with Bessel bound $2^{M+P}B_0^2S$, so we only need to verify the lower frame  inequality.  From  \eqref{no}, for each $\bold{Z}\in \ell^2(\Lambda,\mathbb{C}^S)$, we have
\begin{align}\label{rabit}
4N\sum\limits_{\substack{ m\in T_1,j\in T_2\\  \lambda\in \Lambda}}\Big|\langle \bold{Z},E_{\frac{m}{M}}R_{2N\lambda}\bold{W}_j\rangle\Big|^2&=\sum\limits_{m\in T_1}\int\limits_{0}^{\frac{1}{4N}}
\Big|\Big|\sum\limits_{k=1}^{S}\mathcal{M}_{m,k}^*(\xi)\mathcal{V}_{\bold{Z}_k}(\xi)   \Big|\Big|_{\mathbb{C}^{2(P+1)}}^2\,d\xi,
\end{align}
where $\mathcal{V}_{\bold{Z}_k}(\xi)$ as defined in \eqref{hari}.
Therefore,  by \eqref{mama} and \eqref{rabit}, we have
\begin{align*}
\sum\limits_{\substack{ m\in T_1,j\in T_2\\  \lambda\in \Lambda}}\Big|\langle \bold{Z},E_{\frac{m}{M}}R_{2N\lambda}\bold{W}_j\rangle\Big|^2&\ge A_0\int\limits_{0}^{\frac{1}{4N}} \sum\limits_{k=1}^S||\mathcal{V}_{\bold{Z}_k}(\xi)||_{\mathbb{C}^{4N}}^2\,d\xi\\
&=A_0||\mathcal{F}(\bold{Z})||_{L^2}^2\\
&=A_0||\bold{Z}||_{\ell^2}^2\,\,,
\end{align*}
which completes the proof.
\end{proof}

We conclude this section with an  applicative example of   Theorem \ref{main1}.
\begin{exa}\label{exa}
	Let $ N=2$, $r=1$, $S=2$, $M=2$, $P=7$. Then,  $\Lambda=\{0,\frac{1}{2}\}+2\mathbb{Z},\,  \Omega=[0,\frac{1}{2}[ \cup [1,\frac{3}{2}[,\, T_1=\{0,1\}$, $T_2=\{ 0,1,\dots,7\}$ 	 and take the window sequences $\{\bold{W}_j\}_{j\in T_2} \subset \ell^2(\Lambda,\mathbb{C}^2)$ as in Example \ref{bol}. Then,
\begin{align*}
 \Big|\Big|\mathcal{F}(E_{\frac{m}{M}}\bold{W}_j)(\xi)\Big|\Big|_{\mathbb{C}^S}\le 2 \  \text{for}  \  m \in T_1, \ j \in T_2;  \  \text{a.e.} \, \xi \in \Omega,
\end{align*}
and the system $\mathcal{G}(\Lambda,  T_1, T_2, \bold{W}_j)=\big \{E_{\frac{m}{2}}R_{4\lambda}\bold{W}_j:m\in T_1,j\in T_2,\lambda \in \{0,\frac{1}{2}\}+2\mathbb{Z}  \big \}$ is  DVNUG Bessel sequence with Bessel bound $2^{12}$.

	For $m\in T_1$ and $k\in \{1,2\}$, let $\mathcal{M}_{m,k}(\xi)$ be the corresponding  $8\times 16$ matrices as defined in \eqref{M}. Using any matrix calculating software,  it can be easily checked that for each $m\in T_1=\{0,1\}$, we have
	\begin{align}
	\mathcal{M}_{m,k}(\xi)\mathcal{M}^*_{m,k'}(\xi)=16\delta_{kk'}\bold{I}_8,\,\,\,\,\,\text{for}\,\,1\le k,k'\le 2.
	\end{align}  	
	Now,  for any $\bold{C}_1,\bold{C}_2\in \mathbb{C}^8$ and $a.e.\,\xi\in (0,\frac{1}{8})$, we compute
	\begin{align}\label{trump}
&\sum\limits_{m \in T_1} \Big|\Big|\sum\limits_{k=1}^2 \mathcal{M}^*_{m,k}(\xi)\bold{C}_k\Big|\Big|_{\mathbb{C}^{16}}^2 \nonumber\\
&=\Big|\Big|\sum\limits_{k=1}^2 \mathcal{M}^*_{0,k}(\xi)\bold{C}_k\Big|\Big|_{\mathbb{C}^{16}}^2+\Big|\Big|\sum\limits_{k=1} \mathcal{M}^*_{1,k}(\xi)\bold{C}_k\Big|\Big|_{\mathbb{C}^{16}}^2\nonumber\\
&=\Big\langle \sum\limits_{k=1}^2 \mathcal{M}^*_{0,k}(\xi)\bold{C}_k, \sum\limits_{k'=1}^2 \mathcal{M}^*_{0,k'}(\xi)\bold{C}_{k'} \Big\rangle
+\Big\langle \sum\limits_{k=1}^2 \mathcal{M}^*_{1,k}(\xi)\bold{C}_k, \sum\limits_{k'=1}^2 \mathcal{M}^*_{1,k'}(\xi)\bold{C}_{k'} \Big\rangle\nonumber\\
&=\langle 16\bold{I}_8\bold{C}_1,\bold{C}_1\rangle+\langle 16\bold{I}_8\bold{C}_2,\bold{C}_2\rangle + \langle 16\bold{I}_8\bold{C}_1,\bold{C}_1\rangle+\langle 16\bold{I}_8\bold{C}_2,\bold{C}_2\rangle\nonumber\\
&=16\big(||\bold{C}_1||^2+||\bold{C}_2||^2+||\bold{C}_1||^2+||\bold{C}_2||^2\big)\nonumber\\
&=32\sum\limits_{k=1}^2||\bold{C}_k||^2.
 	 	 	\end{align}
Hence,  by \eqref{trump} and Theorem \ref{main1},  $\mathcal{G}(\Lambda,  T_1, T_2, \bold{W}_j)$   is a  DVNUDG frame with bounds $4$ and $2^{12}$.
\end{exa}

\section{Perturbation of  Discrete Vector-Valued Nonuniform Gabor Frames}\label{sect4}
Perturbation theory is one of important branch of both pure mathematics and physical science, which have been studied extensively,  see \cite{Kato} for technical details. Under perturbation of frames, it is important that their fundamental properties are preserved. Some fundamental results about  perturbation  of Gabor frames can be found in \cite{ole, FZUS}. The following result shows that   DVNUG frames are stable under small perturbation.
\begin{thm}\label{per}
	Let $\mathcal{G}(\Lambda,  T_1, T_2, \bold{W}_j)$ be a  DVNUG frame for $\ell^2(\Lambda,\mathbb{C}^S)$ with  frames bounds $A_0$ and $B_0$. Let $\{\bold{V}_j\}_{j\in T_2}\subset \ell^2(\Lambda,\mathbb{C}^S)$,  and $\theta$ be a positive  real constant satisfying
	\begin{align}\label{topu}
	\Big|\Big|\mathcal{F}E_{\frac{m}{M}}(\bold{W}_j+ \bold{V}_j)(\xi)\Big|\Big|_{\mathbb{C}^S}\le \theta< 2^{M+P}\theta^2 S <A_0,
	\end{align}
	for $m\in T_1$, $j\in T_2$,  $a.e.\,\xi\in \Omega$. Then,  $\mathcal{G}(\Lambda,  T_1, T_2, \bold{V}_j)$
  is also a DVNUG frame for $\ell^2(\Lambda,\mathbb{C}^S)$ with bounds $\big(\sqrt{A_0}-\sqrt{2^{(M+P)}\theta^2 S }\big)^2$ and $\big(2^{(M+P+1)}\theta^2 S +2B_0\big)$.	
		\end{thm}
	\begin{proof}
For any  $\bold{Z}\in \ell^2(\Lambda,\mathbb{C}^S)$, using Lemma \ref{lemmahari}, we have
\begin{align}
&\sum\limits_{\substack{ m\in T_1,j\in T_2\\  \lambda\in \Lambda}}\Big|\langle \bold{Z},E_{\frac{m}{M}}R_{2N\lambda}\bold{V}_j\rangle\Big|^2 \nonumber\\
&=\sum\limits_{\substack{ m\in T_1,j\in T_2\\  \lambda\in \Lambda}}\Big|\langle \bold{Z},E_{\frac{m}{M}}R_{2N\lambda}(\bold{V}_j+\bold{W}_j)\rangle-\langle \bold{Z},E_{\frac{m}{M}}R_{2N\lambda} \bold{W}_j\rangle\Big|^2 \label{godimedia1}\\
&\le 2\Bigg( \sum\limits_{\substack{ m\in T_1,j\in T_2\\  \lambda\in \Lambda}}\Big|\langle \bold{Z},E_{\frac{m}{M}}R_{2N\lambda}(\bold{V}_j+\bold{W}_j)\rangle\Big|^2 +\sum\limits_{\substack{ m\in T_1,j\in T_2\\  \lambda\in \Lambda}}\Big|\langle \bold{Z},E_{\frac{m}{M}}R_{2N\lambda} \bold{W}_j\rangle\Big|^2\Bigg)\label{56}.
\end{align}

  Also, by \eqref{topu} and Theorem \ref{th1}, we have
\begin{align}\label{jaja}
\sum\limits_{\substack{ m\in T_1,j\in T_2\\  \lambda\in \Lambda}}\Big|\langle \bold{Z},E_{\frac{m}{M}}R_{2N\lambda}(\bold{W}_j+\bold{V}_j)\rangle\Big|^2 \le 2^{(M+P)}\theta^2S||\bold{Z}||^2.
\end{align}
	 Now, by the given hypothesis, \eqref{56}  and \eqref{jaja}, we obtain
	 \begin{align}\label{popu}
	 \sum\limits_{\substack{ m\in T_1,j\in T_2\\  \lambda\in \Lambda}}\Big|\langle \bold{Z},E_{\frac{m}{M}}R_{2N\lambda}\bold{V}_j\rangle\Big|^2\le \big(2^{(M+P+1)}\theta^2 S +2B_0\big)||\bold{Z}||^2.
	 \end{align}
By invoking inequalities \eqref{godimedia1} and \eqref{jaja}, we  have 	
\begin{align}\label{gopu}
&\sqrt{\sum\limits_{\substack{ m\in T_1,j\in T_2\\  \lambda\in \Lambda}}\Big|\langle \bold{Z},E_{\frac{m}{M}}R_{2N\lambda}\bold{V}_j\rangle\Big|^2} \nonumber\\
&=\sqrt{\sum\limits_{\substack{ m\in T_1,j\in T_2\\  \lambda\in \Lambda}}\Big|\langle \bold{Z},E_{\frac{m}{M}}R_{2N\lambda}(\bold{V}_j+\bold{W}_j)\rangle-\langle \bold{Z},E_{\frac{m}{M}}R_{2N\lambda} \bold{W}_j\rangle\Big|^2}\nonumber\\	
&\ge \sqrt{\sum\limits_{\substack{ m\in T_1,j\in T_2\\  \lambda\in \Lambda}}\Big|\langle \bold{Z},E_{\frac{m}{M}}R_{2N\lambda} \bold{W}_j\rangle\Big|^2}
-\sqrt{\sum\limits_{\substack{ m\in T_1,j\in T_2\\  \lambda\in \Lambda}}\Big|\langle \bold{Z},E_{\frac{m}{M}}R_{2N\lambda}(\bold{V}_j+\bold{W}_j)\Big|^2}\nonumber\\	
&\ge \sqrt{A_0||\bold{Z}||^2}-\sqrt{(2^{(M+P)}\theta^2S)||\bold{Z}||^2}\nonumber\\
&=\big(\sqrt{A_0}-\sqrt{2^{(M+P)}\theta^2 S }\big)||\bold{Z}||, \ \ \bold{Z}\in \ell^2(\Lambda,\mathbb{C}^S).	
\end{align}
From   \eqref{popu} and  \eqref{gopu}, we conclude that   $\mathcal{G}(\Lambda,  T_1, T_2, \bold{V}_j)$ is a DVNUG frame for $\ell^2(\Lambda,\mathbb{C}^S)$ with the desired frame bounds.
\end{proof}
In the end of this section, we illustrate  Theorem \ref{per} with an example.
\begin{exa}
Let $ N=2$, $r=1$, $S=2$, $M=2$, $P=7$. Then, $\Lambda=\{0,\frac{1}{2}\}+2\mathbb{Z}$,  $\Omega=[0,\frac{1}{2}[ \cup [1,\frac{3}{2}[$, $ T_1=\{0,1\}$, $T_2=\{ 0,1,\dots,7\}$  and $\mathcal{G}(\Lambda,  T_1, T_2, \bold{W}_j)$  be a DVNUG frame with bounds $A_0=4$ and $B_0=2^{12}$, as given in Example \ref{exa}. Define $\{\bold{V}_j\}_{j\in T_2}\subset \ell^2(\Lambda,\mathbb{C}^2)$ as follows.
\begin{align*}
&\bold{V}_0(0)=\begin{bmatrix}
-16/17\\
0
\end{bmatrix},\,\,\,\, \bold{V}_0(4)=\begin{bmatrix}
-1\\
0
\end{bmatrix},\,\,\,\,  \bold{V}_0(\lambda)=\begin{bmatrix}
0\\
0
\end{bmatrix} \,\text{for}\,\lambda\in \Lambda\setminus\{0,4\};\\
&\bold{V}_1(0)=\begin{bmatrix}
-16/17\\
0
\end{bmatrix},\, \,\,\,  \bold{V}_1(4)=\begin{bmatrix}
1\\
0
\end{bmatrix},\,\,\,\,  \bold{V}_1(\lambda)=\begin{bmatrix}
0\\
0
\end{bmatrix} \,\text{for}\,\lambda\in \Lambda\setminus\{0,4\};\\
&\bold{V}_2(0)=\begin{bmatrix}
0\\
-16/17
\end{bmatrix},\, \,\,\,  \bold{V}_2(4)=\begin{bmatrix}
0\\
-1
\end{bmatrix},\,\,\,\, \bold{V}_2(\lambda)=\begin{bmatrix}
0\\
0
\end{bmatrix} \,\text{for}\,\lambda\in \Lambda\setminus\{0,4\};\\
&\bold{V}_3(0)=\begin{bmatrix}
0\\
-16/17
\end{bmatrix},\,  \,\,\,\,  \bold{V}_3(4)=\begin{bmatrix}
0\\
1
\end{bmatrix},\,\,\,\, \bold{V}_3(\lambda)=\begin{bmatrix}
0\\
0
\end{bmatrix}\,\text{for}\,\lambda\in \Lambda\setminus\{0,4\};\\
&\bold{V}_4\Big(\frac{1}{2}\Big)=\begin{bmatrix}
-1\\
0
\end{bmatrix},\, \,\,\,  \bold{V}_4\Big(\frac{1}{2}+4\Big)=\begin{bmatrix}
-1\\
0
\end{bmatrix},\,\,\,\, \bold{V}_4(\lambda)=\begin{bmatrix}
0\\
0
\end{bmatrix} \,\text{for}\,\lambda\in \Lambda\setminus\Big\{\frac{1}{2},\frac{1}{2}+4\Big\};\\
&\bold{V}_5\Big(\frac{1}{2}\Big)=\begin{bmatrix}
-1\\
0
\end{bmatrix},\,  \bold{V}_5\Big(\frac{1}{2}+4\Big)=\begin{bmatrix}
1\\
0
\end{bmatrix},\,\,\,\, \bold{V}_5(\lambda)=\begin{bmatrix}
0\\
0
\end{bmatrix} \,\text{for}\,\lambda\in \Lambda\setminus\Big\{\frac{1}{2},\frac{1}{2}+4\Big\};\\
&\bold{V}_6\Big(\frac{1}{2}\Big)=\begin{bmatrix}
0\\
-1
\end{bmatrix},\,  \bold{V}_6\Big(\frac{1}{2}+4\Big)=\begin{bmatrix}
0\\
-1
\end{bmatrix},\,\,\,\, \bold{V}_6(\lambda)=\begin{bmatrix}
0\\
0
\end{bmatrix} \,\text{for}\,\lambda\in \Lambda\setminus\Big\{\frac{1}{2},\frac{1}{2}+4\Big\};\\
&\bold{V}_7\Big(\frac{1}{2}\Big)=\begin{bmatrix}
0\\
-1
\end{bmatrix},\, \,\,  \bold{V}_7\Big(\frac{1}{2}+4\Big)=\begin{bmatrix}
0\\
1
\end{bmatrix},\,\,\,\, \bold{W}_7(\lambda)=\begin{bmatrix}
0\\
0
\end{bmatrix} \,\text{for}\,\lambda\in \Lambda\setminus\Big\{\frac{1}{2},\frac{1}{2}+4\Big\}.
\end{align*}	
Now for $m\in T_1=\{0,1 \},j\in T_2  $, the system  $E_{\frac{m}{2}}\bold{W}_j=\big\{(E_{\frac{m}{2}}\bold{W}_j)(\lambda)\big\}_{\lambda\in \Lambda} $
are as follows:
\begin{align*}
&\big(E_{\frac{m}{2}} \bold{V}_0\big)(0)=\begin{bmatrix}
-16/17\\
0
\end{bmatrix},\, \,\,\,\,  \big(E_{\frac{m}{2}} \bold{V}_0\big)(4)=\begin{bmatrix}
-1\\
0
\end{bmatrix},\,\,\,\,  \big(E_{\frac{m}{2}} \bold{V}_0\big)(\lambda)=\begin{bmatrix}
0\\
0
\end{bmatrix}\,\text{for}\,\lambda\in \Lambda\setminus\{0,4\};\\
&\big(E_{\frac{m}{2}} \bold{V}_1\big)(0)=\begin{bmatrix}
-16/17\\
0
\end{bmatrix},\, \,\,\, \big(E_{\frac{m}{2}} \bold{V}_1\big)(4)=\begin{bmatrix}
1\\
0
\end{bmatrix},\,\, \,\, \big(E_{\frac{m}{2}} \bold{V}_1\big)(\lambda)=\begin{bmatrix}
0\\
0
\end{bmatrix}\,\text{for}\,\lambda\in \Lambda\setminus\{0,4\};\\
&\big(E_{\frac{m}{2}} \bold{V}_2\big)(0)=\begin{bmatrix}
0\\
-16/17
\end{bmatrix},\, \,\,\,  \big(E_{\frac{m}{2}} \bold{V}_2\big)(4)=\begin{bmatrix}
0\\
-1
\end{bmatrix},\,\, \,\,  \big(E_{\frac{m}{2}} \bold{V}_2\big)(\lambda)=\begin{bmatrix}
0\\
0
\end{bmatrix} \,\text{for}\,\lambda\in \Lambda\setminus\{0,4\};\\
&\big(E_{\frac{m}{2}} \bold{V}_3\big)(0)=\begin{bmatrix}
0\\
-16/17
\end{bmatrix},\, \,\,\,\, \big(E_{\frac{m}{2}} \bold{V}_3\big)(4)=\begin{bmatrix}
0\\
1
\end{bmatrix},\,\,\,\, \big(E_{\frac{m}{2}} \bold{V}_3\big)(\lambda)=\begin{bmatrix}
0\\
0
\end{bmatrix}\,\text{for}\,\lambda\in \Lambda\setminus\{0,4\};\\
&\big(E_{\frac{m}{2}} \bold{V}_4\big)\Big(\frac{1}{2}\Big)=\begin{bmatrix}
-e^{\pi i \frac{m}{2}}\\
0
\end{bmatrix},  \big(E_{\frac{m}{2}} \bold{V}_4\big)\Big(\frac{1}{2}+4\Big)=\begin{bmatrix}
-e^{\pi i \frac{m}{2}}\\
0
\end{bmatrix}, \,\, \,\,  \big(E_{\frac{m}{2}} \bold{V}_4\big)(\lambda)=\begin{bmatrix}
0\\
0
\end{bmatrix}\,\text{for}\,\lambda\in \Lambda\setminus\Big\{\frac{1}{2},\frac{1}{2}+4\Big\};\\
&\big(E_{\frac{m}{2}} \bold{V}_5\big)\Big(\frac{1}{2}\Big)=\begin{bmatrix}
-e^{\pi i \frac{m}{2}}\\
0
\end{bmatrix}, \big(E_{\frac{m}{2}} \bold{V}_5\big)\Big(\frac{1}{2}+4\Big)=\begin{bmatrix}
e^{\pi i \frac{m}{2}}\\
0
\end{bmatrix},\,\, \,\, \big(E_{\frac{m}{2}} \bold{V}_5\big)(\lambda)=\begin{bmatrix}
0\\
0
\end{bmatrix} \,\text{for}\,\lambda\in \Lambda\setminus\Big\{\frac{1}{2},\frac{1}{2}+4\Big\};\\
&\big(E_{\frac{m}{2}} \bold{V}_6\big)\Big(\frac{1}{2}\Big)=\begin{bmatrix}
0\\
-e^{\pi i \frac{m}{2}}
\end{bmatrix},\big(E_{\frac{m}{2}} \bold{V}_6\big)\Big(\frac{1}{2}+4\Big)=\begin{bmatrix}
0\\
-e^{\pi i \frac{m}{2}}
\end{bmatrix},  \,\, \,\,  \big(E_{\frac{m}{2}} \bold{V}_6\big)(\lambda)=\begin{bmatrix}
0\\
0
\end{bmatrix} \,\text{for}\,\lambda\in \Lambda\setminus\Big\{\frac{1}{2},\frac{1}{2}+4\Big\};\\
&\big(E_{\frac{m}{2}} \bold{V}_7\big)\Big(\frac{1}{2}\Big)=\begin{bmatrix}
0\\
-e^{\pi i \frac{m}{2}}
\end{bmatrix},\big(E_{\frac{m}{2}} \bold{V}_7\big)\Big(\frac{1}{2}+4\Big)=\begin{bmatrix}
0\\
e^{\pi i \frac{m}{2}}
\end{bmatrix},   \,\, \,\, \big(E_{\frac{m}{2}} \bold{V}_7\big)(\lambda)=\begin{bmatrix}
0\\
0
\end{bmatrix} \,\text{for}\,\lambda\in \Lambda\setminus\Big\{\frac{1}{2},\frac{1}{2}+4\Big\}.
\end{align*}	
And, for $m\in T_1$ and $j\in T_2$, the corresponding Fourier transforms $\mathcal{F}(E_{\frac{m}{M}}\bold{W}_j)(\xi)$
are as follows:	
\begin{align*}
&\mathcal{F}(E_{\frac{m}{M}}\bold{V}_0)(\xi)=\begin{bmatrix}
-16/17-e^{8\pi i \xi}\\
0
\end{bmatrix}, &\mathcal{F}(E_{\frac{m}{M}}\bold{V}_1)(\xi)=\begin{bmatrix}
-16/17+e^{8\pi i \xi}\\
0
\end{bmatrix};\\ &\mathcal{F}(E_{\frac{m}{M}}\bold{V}_2)(\xi)=\begin{bmatrix}
0\\
-16/17-e^{8\pi i \xi}
\end{bmatrix},
&\mathcal{F}(E_{\frac{m}{M}}\bold{V}_3)(\xi)=\begin{bmatrix}
0\\
-16/17+e^{8\pi i \xi}
\end{bmatrix};\\ &\mathcal{F}(E_{\frac{m}{M}}\bold{V}_4)(\xi)=\begin{bmatrix}
-e^{\pi i\frac{m}{2}}(e^{\pi i \xi}+e^{9\pi i\xi})\\
0
\end{bmatrix}, &\mathcal{F}(E_{\frac{m}{M}}\bold{V}_5)(\xi)=\begin{bmatrix}
e^{\pi i\frac{m}{2}}(e^{9\pi i\xi}-e^{\pi i \xi})\\
0
\end{bmatrix};\\
&\mathcal{F}(E_{\frac{m}{M}}\bold{V}_6)(\xi)=\begin{bmatrix}
0\\
-e^{\pi i\frac{m}{2}}(e^{\pi i \xi}+e^{9\pi i\xi})
\end{bmatrix},\,\,&\mathcal{F}(E_{\frac{m}{M}}\bold{V}_7)(\xi)=\begin{bmatrix}
0\\
e^{\pi i\frac{m}{2}}(e^{9\pi i\xi}-e^{\pi i \xi})
\end{bmatrix}.	
\end{align*}
It is easy to verify for $m\in T_1$, $j\in T_2,$ and $\xi \in \Omega$ that,
\begin{align*}
	\Big|\Big|\mathcal{F}E_{\frac{m}{M}}(\bold{W}_j + \bold{V}_j)(\xi)\Big|\Big|_{\mathbb{C}^2}&=\Big|\Big|\mathcal{F}(E_{\frac{m}{M}}\bold{W}_j(\xi))+ \mathcal{F}(E_{\frac{m}{M}}\bold{V}_j)(\xi))\Big|\Big|_{\mathbb{C}^2}\le 1/17=\theta,
\end{align*}
and $2^{7+2}\times\theta^2 \times 2=3.54 <4=A_0   $. Hence,  by Theorem \ref{per}, $\mathcal{G}(\Lambda,  T_1, T_2, \bold{V}_j)$ is a  DVNUG frame for the space $\ell^2(\Lambda,\mathbb{C}^2)$.
	
\end{exa}
\section{Interplay Between Window Sequences of Discrete Vector-Valued  Nonuniform System and Their Corresponding Coordinates}\label{sect5}	
In real world,  many times it is convenient to work with compact data set due to various reasons like incomplete data points or the given data set  is so huge and the each particual detail of the  given data is not so much important, for example if we have data set of last one thousand days of temperature $\bold{T}_d$ which is calculated in three times in a day, like morning $m_d$ , afternoon $a_d$ and night $n_d$, where $d$ denotes the day, that is,  $\bold{T}_d=\begin{bmatrix}
m_d\\
a_d\\
n_d
\end{bmatrix},
$
and we are working in a problem where we need last thousand days's temperature but not specifically three times in a day, or may be in our data points some entries are missing, then we can work with  average data points of temperature $\text{\larger[2]$\mu$}_{\bold{T}_d}=\frac{1}{3}(m_d+a_d+n_d) $. The problem becomes more critical if we have nonuniform data set for example if we are taking temperature of each 2 hours and half an hour later of each 2 hours,    i.e.,  $2k, 2k+\frac{1}{2}, k=0,1,\dots$. This motivate us to find frame condition for arithmetic mean of window sequences associated with a given DVNUG frame for the discrete vector-valued nonuniform signal space  $\ell^2(\Lambda,\mathbb{C}^S)$. In this direction, we have the following result.
\begin{thm}\label{InterthI}
Let $\{\bold{W}_j\}_{j\in T_2}\in \ell^2(\Lambda,\mathbb{C}^S)$  be such that  $\mathcal{G}(\Lambda,  T_1, T_2, \bold{W}_j)$ is a  DVNUG frame for $\ell^2(\Lambda,\mathbb{C}^S)$ with frame bounds $A_0$ and $B_0$. Then, $\mathcal{G}(\Lambda,  T_1, T_2, \text{\larger[3]$\mu$}_{\bold{W}_j})$ is DVNUG frame of $\ell^2(\Lambda,\mathbb{C})$ with frame  bounds $\frac{A_0}{S}$ and $\frac{B_0}{S}$ .
\end{thm}
\begin{proof}
	Let $x=\{x(\lambda)\}_{\lambda\in \Lambda}\in \ell^2(\Lambda,\mathbb{C}) $ be an arbitrary element. Define $\bold{Z}\in \ell^2(\Lambda,\mathbb{C}^S)$ as follows
	\begin{align*}
	\bold{Z}=\Bigg\{ \begin{bmatrix}
	[\bold{Z}(\lambda)]_1
	\vspace{5pt}\\
	[\bold{Z}(\lambda)]_2\\
	\vdots
	\vspace{5pt}\\
	[\bold{Z}(\lambda)]_S
	\end{bmatrix}  \Bigg\}_{\lambda\in \Lambda}=\Bigg\{ \begin{bmatrix}
	x(\lambda)
	\vspace{5pt}\\
	x(\lambda)\\
	\vdots
	\vspace{5pt}\\
	x(\lambda)
	\end{bmatrix}  \Bigg\}_{\lambda\in \Lambda}.
	\end{align*}
Clearly, $||\bold{Z}||_{\ell^2}=\sqrt{S}||x||_{\ell^2}$.  Since $\mathcal{G}(\Lambda,  T_1, T_2, \bold{W}_j)$ is a  DVNUG farme with bounds $A_0$ and $B_0$,  we have
	\begin{align}\label{lala}
	SA_0||x||^2\le \sum\limits_{\substack{ m\in T_1,j\in T_2\\  \lambda\in \Lambda}}\Big|\langle \bold{Z},E_{\frac{m}{M}}R_{2N\lambda}\bold{W}_j\rangle\Big|^2\le SB_0||x||^2, \ x \in \ell^2(\Lambda,\mathbb{C}).
	\end{align}
	Now consider,
	\begin{align}\label{lala2}
	\sum\limits_{\substack{ m\in T_1,j\in T_2\\  \lambda\in \Lambda}}\Big|\langle \bold{Z},E_{\frac{m}{M}}R_{2N\lambda}\bold{W}_j\rangle\Big|^2&=S^2\sum\limits_{\substack{ m\in T_1,j\in T_2\\  \lambda\in \Lambda}}\Bigg|\sum\limits_{\substack{ k=1,  \lambda'\in \Lambda}}^S[\bold{Z}(\lambda')]_k
	\frac{1}{S}\overline{E_{\frac{m}{M}}R_{2N\lambda}[\bold{W}_j(\lambda')]_k}\Bigg|^2\nonumber\\
	&=S^2\sum\limits_{\substack{ m\in T_1,j\in T_2\\  \lambda\in \Lambda}}\Bigg|\sum\limits_{  \lambda'\in \Lambda} x(\lambda')
	\overline{E_{\frac{m}{M}}R_{2N\lambda}\frac{1}{S}\sum\limits_{k=1}^S[\bold{W}_j(\lambda')]_k}\Bigg|^2\nonumber\\
	&=S^2\sum\limits_{\substack{ m\in T_1,j\in T_2\\  \lambda\in \Lambda}}\Bigg|\sum\limits_{  \lambda'\in \Lambda} x(\lambda')\overline{E_{\frac{m}{M}}R_{2N\lambda}\text{\larger[3]$\mu$}_{\bold{W}_j}(\lambda')}\Bigg|^2\nonumber\\
	&=S^2\sum\limits_{\substack{ m\in T_1,j\in T_2\\  \lambda\in \Lambda}}\Big|\langle x,E_{\frac{m}{M}}R_{2N\lambda}\text{\larger[3]$\mu$}_{\bold{W}_j}\rangle\Big|^2.
	\end{align}
	Using \eqref{lala} and \eqref{lala2}, we have
	\begin{align*}
	\frac{A_0}{S}||x||^2\le \sum\limits_{\substack{ m\in T_1,j\in T_2\\  \lambda\in \Lambda}}\Big|\langle \bold{Z},E_{\frac{m}{M}}R_{2N\lambda}\bold{W}_j\rangle\Big|^2\le \frac{B_0}{S}||x||^2,  \ x \in \ell^2(\Lambda,\mathbb{C}).
	\end{align*}
This concludes the result.
\end{proof}
To conclude the section, we give  relationships  between the window sequences of DVNUG system and its corresponding coordinates.
For different types of relations of frames of  matrix-valued  wave packet  systems and its associated atomic wave packets, we refer to  \cite{JV202021}.
\begin{thm}\label{120}
	Let $\{\bold{W}_j\}_{j\in T_2} \subset \ell^2(\Lambda,\mathbb{C}^S) $. Define $S\times (P+1)$  window matrix, $\bold{M}=\begin{bmatrix}
	\bold{W}_0 & \bold{W}_1 \,\, \cdots\,\, \bold{W}_P
	\end{bmatrix} $, that is,
	\begin{align*}
	\bold{M}(\lambda)=\begin{bmatrix}
	\big[\bold{W}_0(\lambda)\big]_1 & \big[\bold{W}_1(\lambda)\big]_1 &\cdots &\big[\bold{W}_P(\lambda)\big]_1
	\vspace{5pt}\\
\big[\bold{W}_0(\lambda)\big]_2 & \big[\bold{W}_1(\lambda)\big]_2 &\cdots &\big[\bold{W}_P(\lambda)\big]_2\\
\vdots &\vdots &\vdots &\vdots
\vspace{5pt}\\
\big[\bold{W}_0(\lambda)\big]_S & \big[\bold{W}_1(\lambda)\big]_S &\cdots &\big[\bold{W}_P(\lambda)\big]_S
	\end{bmatrix}_{S\times (P+1)}\,\,\,\,\,\text{for}\,\,\,\,\lambda\in \Lambda,
	\end{align*}
	where $\big[\bold{M}\big]_{l,l'}=[\bold{W}_{(l'-1)}]_{l}  $ is $l$th row and $l'$th column of matrix $\bold{M}$ with $1\le l \le S, 1\le l'\le P+1$.
	Then the following holds.
	\begin{enumerate}[$(i)$]
\item \label{121}   If  $\mathcal{G}(\Lambda,  T_1, T_2, \bold{W}_j)$  is a  DVNUG frame for $\ell^2(\Lambda,\mathbb{C}^S)$ with bounds $A_0$ and $B_0$, then each row of the window matrix $\bold{M}$ generates a discrete nonuniform Gabor frame for $\ell^2(\Lambda,\mathbb{C})$. To be precise, for each $l_0\in \{1,2,\dots,S\}$, the system
\begin{align*}
 \{E_{\frac{m}{M}}R_{2N\lambda}[\bold{M}]_{l_0l'} :\lambda\in \Lambda,m\in T_1,1\le l'\le P+1  \}
\end{align*}
constitutes  a discrete nonuniform Gabor frame for  $\ell^2(\Lambda,\mathbb{C})$. But, the converse is not true.

\item\label{155} $\mathcal{G}(\Lambda,  T_1, T_2, \bold{W}_j)$ is DVNUG Bessel sequence in  $\ell^2(\Lambda,\mathbb{C}^S)$ if and only if   each  element $[\bold{M}]_{ll'}$ of window matrix $\bold{M}$ generate discrete nonuniform Gabor  Bessel sequence in  $\ell^2(\Lambda,\mathbb{C})$. More precisely,  for each $1\le l \le S, 1\le l'\le P+1$, the system
\begin{align*}
\{E_{\frac{m}{M}}R_{2N\lambda}[\bold{M}]_{ll'} :\lambda\in \Lambda,m\in T_1   \}
\end{align*}
is a discrete nonuniform Gabor  Bessel sequence in  $\ell^2(\Lambda,\mathbb{C})$.
\end{enumerate}
\end{thm}

\begin{proof}
$(i)$: Fix $l_0\in \{1,2,\cdots,S\}$, and let  $z_{l_0}=\{ z_{l_0}(\lambda) \}_{\lambda\in \Lambda}  \in \ell^2(\Lambda,\mathbb{C}) $ be   arbitrary.
Define a  vector $\bold{Z}_0=\Bigg\{\begin{bmatrix}
\big[ \bold{Z}_0(\lambda)  \big]_1
\vspace{5pt}\\
\big[ \bold{Z}_0(\lambda)  \big]_2\\
\vdots\\
\big[ \bold{Z}_0(\lambda)  \big]_S
\end{bmatrix}\Bigg\}_{\lambda\in \Lambda}$ in $\ell^2(\Lambda,\mathbb{C}^S)$
as follows:
\begin{align}\label{hariharan}
\big[ \bold{Z}_0(\lambda)  \big]_k=\begin{cases}
z_{l_0}(\lambda), & \lambda\in \Lambda,k=l_{0},\\
0, & \lambda\in \Lambda, k\in \{1,2,\cdots,S \}\setminus l_0.
\end{cases}
\end{align}
Clearly,  $||\bold{Z}_0||_{\ell^2}=||z_{l_0}||_{\ell^2}$. Let  $A_0$ and $B_0$ be frame bounds for $\mathcal{G}(\Lambda,  T_1, T_2, \bold{W}_j)$. Then,
\begin{align}\label{hariharan1}
A_0||\bold{Z}_0||^2\le \sum\limits_{\substack{ m\in T_1,j\in T_2\\  \lambda\in \Lambda}}\Big|\langle \bold{Z}_0,E_{\frac{m}{M}}R_{2N\lambda}\bold{W}_j\rangle\Big|^2\le B_0||\bold{Z}_0||^2.
\end{align}
Using \eqref{hariharan} and \eqref{hariharan1}, we have
\begin{align}\label{ho1}
A_0||z_{l_0}||^2\le \sum\limits_{\substack{ m\in T_1,j\in T_2\\  \lambda\in \Lambda}}\Big|\langle z_{l_0},E_{\frac{m}{M}}R_{2N\lambda}[\bold{W}_j]_{l_0}\rangle\Big|^2\le B_0||z_{l_0}||^2.
\end{align}
Using the following relation, which is by hypothesis,
\begin{align*}
\{[\bold{W}_j]_{l_0}:j\in T_2 \}=\{[\bold{M}]_{l_0,l'}:1\le l'\le P+1  \},
\end{align*}
and \eqref{ho1}, we conclude that $\{E_{\frac{m}{M}}R_{2N\lambda}[\bold{M}]_{l_0,l'} :\lambda\in \Lambda,m\in T_1,1\le l'\le P+1\}$ is a discrete nonuniform frame for $\ell^2(\Lambda,\mathbb{C})$ with bounds $A_0$ and $B_0$.

To show that converse is not true.  Consider a sequence $\{w_j \}_{j\in T_2} \in \ell^2(\Lambda,\mathbb{C})$  with the property that $\{E_{\frac{m}{M}}R_{2N\lambda}w_j:\lambda\in\Lambda,m\in T_1,j\in T_2 \}$ is a discrete nonuniform  frame for $\ell^2(\Lambda,\mathbb{C})$. Choose $S=3$, and define $\{ \bold{W}_j \}_{j\in T_2}\subset \ell^2(\Lambda,\mathbb{C}^3) $ as follow:
\begin{align*}
\bold{W}_j(\lambda)=\begin{bmatrix}
[\bold{W}_j(\lambda)]_1
\vspace{5pt}\\
[\bold{W}_j(\lambda)]_2
\vspace{5pt}\\
[\bold{W}_j(\lambda)]_3
\end{bmatrix}=\begin{bmatrix}
w_j(\lambda)
\vspace{5pt}\\
w_j(\lambda)
\vspace{5pt}\\
w_j(\lambda)
\end{bmatrix}\,\,\,\,\,\text{for}\,\,\,\,\lambda\in \Lambda, \ j\in T_2.
\end{align*}
Then,  each row of matrix
\begin{align*}
\bold{M}(\lambda)=\begin{bmatrix}
w_0(\lambda)   & w_1(\lambda) &\cdots &w_P(\lambda)
\vspace{5pt}\\
w_0(\lambda)   & w_1(\lambda) &\cdots &w_P(\lambda)
\vspace{5pt}\\
w_0(\lambda)   & w_1(\lambda) &\cdots &w_P(\lambda)
\end{bmatrix}_{3\times (P+1)},\,\,\,\,\,\text{for}\,\,\,\,\lambda\in \Lambda;
\end{align*}
that is, for each $l_0\in \{1,2,3\}$, the sequence  $\{E_{\frac{m}{M}}R_{2N\lambda}[\bold{M}]_{l_0l'} :\lambda\in \Lambda,m\in T_1,1\le l'\le P+1\}$ is a discrete nonuniform Gabor frame for $\ell^2(\Lambda,\mathbb{C})$.

But, $\mathcal{G}(\Lambda,  T_1, T_2, \bold{W}_j)$ is not a DVNUG frame for $\ell^2(\Lambda,\mathbb{C}^3) $. Indeed, let $\alpha_0$ and $\beta_0$ be frame bounds for $\mathcal{G}(\Lambda,  T_1, T_2, \bold{W}_j)$. Then, for  all $\bold{Z}\in \ell^2(\Lambda,\mathbb{C}^3)$, we have
\begin{align}\label{tata}
\alpha_0||\bold{Z}||^2\le \sum\limits_{\substack{ m\in T_1,j\in T_2\\  \lambda\in \Lambda}}\Big|\langle \bold{Z},E_{\frac{m}{M}}R_{2N\lambda}\bold{W}_j\rangle\Big|^2\le \beta_0||\bold{Z}||^2.
\end{align}

\begin{align*}
\text{Choose} \ \bold{Z}_0=\Bigg\{\begin{bmatrix}
[\bold{Z}_0(\lambda)]_1
\vspace{5pt}\\
[\bold{Z}_0(\lambda)]_2
\vspace{5pt}\\
[\bold{Z}_0(\lambda)]_3
\vspace{5pt}
\end{bmatrix} \Bigg\}_{\lambda\in \Lambda} \in \ell^2(\Lambda,\mathbb{C}^3), \ \text{where} \
\bold{Z}_0(\lambda)=\begin{bmatrix}
[\bold{Z}_0(\lambda)]_1
\vspace{5pt}\\
[\bold{Z}_0(\lambda)]_2
\vspace{5pt}\\
[\bold{Z}_0(\lambda)]_3
\vspace{5pt}
\end{bmatrix}=\begin{bmatrix}
z(\lambda)
\vspace{5pt}\\
-z(\lambda)
\vspace{5pt}\\
0
\vspace{5pt}
\end{bmatrix},\,\,\,\,\,\,\lambda\in \Lambda.	
\end{align*}
Here,   $0 \ne z=z(\lambda)\in \ell^2(\Lambda,\mathbb{C})$ is any fixed element. Then,   $\bold{Z}_0$ is a non-zero vector such that
\begin{align*}
\sum\limits_{\substack{ m\in T_1,j\in T_2\\  \lambda\in \Lambda}}\Big|\langle \bold{Z}_0,E_{\frac{m}{M}}R_{2N\lambda}\bold{W}_j\rangle\Big|^2
=\sum\limits_{\substack{ m\in T_1,j\in T_2\\  \lambda\in \Lambda}}\Bigg|\sum\limits_{\substack{ k=1\\  \lambda'\in \Lambda}}^3\big[\bold{Z}_0(\lambda')\big]_k  \overline{\big[E_{\frac{m}{M}}R_{2N\lambda}\bold{W}_j(\lambda') \big]_k}\Bigg|^2 =0,
\end{align*}
which contradicts lower inequality of  \eqref{tata}.

$(ii)$: Firstly, assume that $\mathcal{G}(\Lambda,  T_1, T_2, \bold{W}_j)$  is a DVNUG  Bessel sequence for $\ell^2(\Lambda,\mathbb{C}^S)$ with Bessel bound $\beta$. Let $x=\{x(\lambda)\}_{\lambda\in \Lambda}\in \ell^2(\Lambda,\mathbb{C})$ be an arbitrary element. Let $l_0,l_0'$ are any fixed integers such that  $1\le l_0\le S, 1\le l_0'\le P+1$.

 Define $\bold{Z}_0=\Bigg\{
\begin{bmatrix}
[\bold{Z}_0(\lambda)]_1
\vspace{5pt}\\
[\bold{Z}_0(\lambda)]_2
\vspace{5pt}\\
\vdots\\
[\bold{Z}_0(\lambda)]_S
\end{bmatrix}\Bigg\}_{\lambda\in \Lambda}   \in \ell^2(\Lambda,\mathbb{C}^S)$ as follows
\begin{align*}
[\bold{Z}_0(\lambda)]_k=\begin{cases}
x(\lambda), & \lambda\in \Lambda,k=l_0,\\
0 , & \lambda\in \Lambda,k\in \{1,2,\dots,S\}\setminus l_0
\end{cases}.
\end{align*}
We compute
\begin{align*}
\sum\limits_{\substack{ m\in T_1,\lambda\in \Lambda  }}\Big|\langle x,E_{\frac{m}{M}}R_{2N\lambda}[\bold{M}]_{l_0l_0'}\rangle\Big|^2&=\sum\limits_{\substack{ m\in T_1,  \lambda\in \Lambda}}\Big|\langle x,E_{\frac{m}{M}}R_{2N\lambda}[\bold{W}_{l_0'-1}]_{l_0}\rangle\Big|^2\\
&=\sum\limits_{\substack{ m\in T_1,  \lambda\in \Lambda}}\Big|\langle \bold{Z}_0,E_{\frac{m}{M}}R_{2N\lambda}\bold{W}_{l_0'-1}\rangle\Big|^2\\
&\le \sum\limits_{\substack{ m\in T_1,j\in T_2\\  \lambda\in \Lambda}}\Big|\langle \bold{Z}_0,E_{\frac{m}{M}}R_{2N\lambda}\bold{W}_{j}\rangle\Big|^2\\
&\le \beta ||\bold{Z}_0||^2\\
&= \beta ||x||^2.
\end{align*}

Conversely, assume now that  each element $[\bold{M}]_{ll'}$
of window matrix $\bold{M}$ generates DVNUG Bessel sequence for $\ell^2(\Lambda,\mathbb{C})$, so there exists positive constant $\beta_0$ say, such that  for each $1\le l \le S, 1\le l' \le P+1$,   we have
\begin{align*}
\sum\limits_{\substack{ m\in T_1,  \lambda\in \Lambda}}\Big|\langle x,E_{\frac{m}{M}}R_{2N\lambda}[\bold{M}]_{ll'}\rangle\Big|^2\le \beta_0 ||x||^2,
\end{align*}
for all  $x=\{x(\lambda)\}_{\lambda\in \Lambda}\in \ell^2(\Lambda,\mathbb{C})$, that is,
\begin{align}\label{raman}
\sum\limits_{\substack{ m\in T_1,  \lambda\in \Lambda}}\Big|\sum\limits_{\lambda'\in\Lambda} x(\lambda')\overline{E_{\frac{m}{M}}R_{2N\lambda}[\bold{M}(\lambda')]}_{ll'}\Big|^2\le \beta_0 \sum\limits_{\lambda\in \Lambda}|x(\lambda)|^2.
\end{align}
 Using Lemma \ref{lemmahari}, for any  $\bold{Z}=\Bigg\{   \begin{bmatrix}
  [\bold{Z}(\lambda)]_1
  \vspace{5pt}\\
  [\bold{Z}(\lambda)]_2\\
  \vdots\\
  [\bold{Z}(\lambda)]_S
  \end{bmatrix}
  \Bigg\}_{\lambda\in \Lambda}\in \ell^2(\Lambda,\mathbb{C}^S) $, we compute
 \begin{align}\label{chaman}
  &\sum\limits_{\substack{ m\in T_1,j\in T_2\\  \lambda\in \Lambda}}\Big|\langle \bold{Z},E_{\frac{m}{M}}R_{2N\lambda}\bold{W}_j\rangle\Big|^2 \nonumber\\
  &=\sum\limits_{\substack{ m\in T_1,j\in T_2\\  \lambda\in \Lambda}}\Big|\sum_{\substack{ k=1,  \lambda'\in \Lambda}}^S   [\bold{Z}(\lambda')]_k[\overline{E_{\frac{m}{M}}R_{2N\lambda}\bold{W}_j(\lambda')}]_k\Big|^2\nonumber\\
   &\le \sum\limits_{\substack{ m\in T_1,j\in T_2\\  \lambda\in \Lambda}}\sum\limits_{k=1}^S 2^{S-1}\Big|\sum\limits_{\lambda'\in \Lambda}   [\bold{Z}(\lambda')]_k\overline{E_{\frac{m}{M}}R_{2N\lambda}[\bold{W}_j(\lambda')}]_k\Big|^2 \nonumber\\
   &=2^{S-1}\sum\limits_{j\in T_2}\sum\limits_{k=1}^S\Bigg(\sum\limits_{\substack{ m\in T_1, \lambda\in \Lambda}}\Big|\sum\limits_{\lambda'\in \Lambda}[\bold{Z}(\lambda')]_k\overline{E_{\frac{m}{M}}R_{2N\lambda}[\bold{M}(\lambda')]}_{k,j+1}\Big|^2 \Bigg).
       \end{align}
Inequalities in  \eqref{raman} and \eqref{chaman} gives
\begin{align*}
\sum\limits_{\substack{ m\in T_1,j\in T_2\\  \lambda\in \Lambda}}\Big|\langle \bold{Z},E_{\frac{m}{M}}R_{2N\lambda}\bold{W}_j\rangle\Big|^2 & \le \beta_0 2^{S-1}\sum\limits_{j\in T_2}\sum\limits_{k=1}^S  \sum\limits_{\lambda\in \Lambda}\big|[\bold{Z}(\lambda)]_k\big|^2\\
&=\beta_0 2^{S-1}\sum\limits_{j\in T_2}||\bold{Z}||_{\ell^2}^2\\
&\le \beta_0 2^{S-1}(P+1)||\bold{Z}||_{\ell^2}^2.
\end{align*}
This concludes the proof.
\end{proof}
\begin{rem}
One may observed that result given in	Theorem \ref{120} (\ref{155}) is not true in case of DVNUG frames for $\ell^2(\Lambda,\mathbb{C}^S)$.
\end{rem}

%
%

\end{document}